\documentclass[12pt]{amsart} 
\usepackage{amssymb,latexsym}
\usepackage{enumerate}

\makeatletter
\@namedef{subjclassname@2010}{%
  \textup{2010} Mathematics Subject Classification}
\makeatother

\newtheorem{thm}{Theorem}[section]
\newtheorem{cor}[thm]{Corollary}
\newtheorem{lem}[thm]{Lemma}

\theoremstyle{definition}

\newtheorem{rem}[thm]{Remark}

\numberwithin{equation}{section}

\newcommand{\bC}{{\mathbb C}}
\newcommand{\bK}{{\mathbb K}}

\newcommand{\bQ}{{\mathbb Q}}

\newcommand{\bZ}{{\mathbb Z}}

\newcommand{\cA}{{\mathcal A}}
\newcommand{\cC}{{\mathcal C}}
\newcommand{\cD}{{\mathcal D}}
\newcommand{\cN}{{\mathcal N}}
\newcommand{\cO}{{\mathcal O}}

\newcommand{\core}{{\rm core}}

\begin{document}


\baselineskip=17pt



\title{Thue's Fundamentaltheorem, I: The General Case}

\author{Paul M. Voutier}
\address{London, UK}
\email{Paul.Voutier@gmail.com}

\date{}

\begin{abstract}
In this paper, we examine Thue's Fundamentaltheorem, showing that it includes,
and often strengthens, known effective irrationality measures obtained via the
so-called hypergeometric method as well as showing that it can be applied to
previously unconsidered families of algebraic numbers. Furthermore, we extend
the method to also cover approximation by algebraic numbers in imaginary
quadratic number fields.
\end{abstract}


\maketitle

\section{Introduction}

\subsection{Background}

In the 1840's, Liouville \cite{Liou} established the existence of 
transcendental numbers by actually constructing one. His construction 
was based on his discovery that for any algebraic number $\alpha$ 
of degree $n \geq 2$, there exists a real number $c(\alpha) > 0$ such that 
\begin{displaymath}
\left| \alpha - \frac{p}{q} \right| 
> \frac{c(\alpha)}{|q|^{n}}, 
\end{displaymath}
for all integers $p$ and $q$ with $q \neq 0$. It was this 
work which first demonstrated the now well-established link 
between transcendence and diophantine problems. 

In 1909, Thue \cite{Thue3} improved upon Liouville's diophantine 
result by introducing a method which eventually led, in 1955, 
to Roth's proof \cite{Roth} that for any irrational algebraic 
number $\alpha$ and any $\epsilon > 0$, there exists 
$c(\alpha,\epsilon) > 0$ such that 
\begin{displaymath}
\left| \alpha - \frac{p}{q} \right| 
> \frac{c(\alpha,\epsilon)}{|q|^{2+\epsilon}}, 
\end{displaymath}
for all integers $p$ and $q$ with $q \neq 0$. 

We call the exponents on $|q|$ in these inequalities 
{\it irrationality measures for $\alpha$} and Roth's irrationality
measures are essentially best possible. 

But the reader should not be misled by this phrase `best possible', for here,
as is often the case, there is more to be done. From Liouville's proof it is
possible to explicitly determine the constant, $c(\alpha)$, but this is not
true for the results of Thue or Roth. This is important as an irrationality
measure even slightly less than $n$ along with an explicit constant (such a
result is called {\it effective}) can yield bounds on the size of solutions 
of many classes of diophantine equations.  

At present, there are three methods available for obtaining 
such effective irrationality measures. 

The first is due to Alan Baker, who, in 1964, published two 
papers \cite{Baker1,Baker2} in which he obtained such effective 
irrationality measures for certain algebraic numbers of the 
form $z^{m/n}$. As an example, he showed that for all integers 
$p$ and $q$, with $q \neq 0$, 
\begin{displaymath}
\left| 2^{1/3} - \frac{p}{q} \right| > \frac{10^{-6}}{|q|^{2.955}}. 
\end{displaymath}

Such results, via this technique, have since been improved, notably
through Chudnovsky's analysis of denominators of the coefficients
of certain hypergeometric functions \cite{Chud}. The best result currently
known, from \cite{Vout1}, states that for any integers $p$ and $q$,
with $q \neq 0$, 
\begin{displaymath}
\left| 2^{1/3} - \frac{p}{q} \right| > \frac{0.25}{|q|^{2.4325}}. 
\end{displaymath}

Baker also pioneered the second method. Later in the 1960's, he \cite{Baker3}
established a remarkable result: lower bounds for linear forms in logarithms.
Among the many applications of this result, in a refined form, are effective 
irrationality measures which are better than Liouville's for 
any algebraic number of degree at least three. The reader is 
invited to consult \cite{BS} where effective irrationality 
measures for numbers of the form $\sqrt[3]{n}$ with $n \in \bZ$ 
are established.

Finally, in the early 1980's, Bombieri \cite{Bom} combined 
elements of the non-effective method of Thue and Siegel with a 
result of Dyson, which was itself discovered for such diophantine 
approximation purposes, to create a method which under suitable 
conditions gives rise to effective irrationality measures much 
better than Liouville's. Along with van der Poorten 
and Vaaler, he \cite{BPV} later refined this method in the case 
of numbers which are cubic irrationalities over number fields. 

\subsection{The Present Work}

In this article, we shall consider ideas related to the first method, the basis
of which lies in the work of Thue, his Fundamentaltheorem \cite{Thue4}.
This work was a continuation of his earlier results \cite{Thue1, Thue2} in which
he explicitly determined polynomials $P_{r}(x)$ and $Q_{r}(x)$ such that
$$
Q_{r}(x)x^{1/n}-P_{r}(x)=(x-1)^{2r+1}\overline{S}_{r}(x),
$$
where $\overline{S}_{r}(x)$ is regular at $x=1$.

Siegel \cite{S1} recognised these $P_{r}(x)$ and $Q_{r}(x)$ as hypergeometric
polynomials. He \cite{S2} also recognised that the polynomials, $F(x)$,
satisfying the differential equation in Thue's Fundamentaltheorem are those
given in Lemma~\ref{lem:diff-eqn} for $m=2$.

In earlier papers \cite{Chen1, CV, LPV}, Thue's Fundamentaltheorem was used
to completely solve several families of Thue equations and inequalities. In
this paper, we investigate the precise conditions under which Thue's
Fundamentaltheorem yields effective irrationality measures for algebraic numbers.

As a result, we show that Thue's Fundamentaltheorem includes all the
effective irrationality measures for numbers of the form $z^{1/n}$,
which can be obtained Baker's first method above and its refinements.

But, in addition to that, we also obtain effective irrationality measures
for a new family of algebraic numbers. These results include all the previous
results (\cite{Chen1, CV, LPV, TVW}) derived from Thue's Fundamentaltheorem.

Furthermore, like Yuan \cite{PZ}, we are able to extend our results to
diophantine approximation over imaginary quadratic fields (the only other
number fields besides $\bQ$, that possess the property of ``discreteness''
of its integers).

However, there are some related tools that are not dealt with here. In particular,
it is possible to use Pad\'{e} approximations to several functions simultaneously
to obtain effective irrationality measures (see \cite{Chud}). A striking example
of this technique is Bennett's paper \cite{Ben}, in which it is used to obtain
effective irrationality measures for numbers of the form $(b/a)^{1/n}$, where
$a$ and $b$ are ``small'' rational integers. These cannot be treated by the
usual ``non-simultaneous'' technique.

See also Wakabayashi's papers \cite{Waka1, Waka2} where simultaneous Pad\'{e}
approximations to the functions $\sqrt{1-a_{1}x}$ and $\sqrt{1-a_{2}x}$ are used
to obtain effective irrationality measures for the real roots of some families
of polynomials of the form $x^{4}-a^{2}x^{2}+b$. These roots are not covered by
our results here.

Finally, in our notation below, we require that $W(x)$ is near $1$. In a
forthcoming paper \cite{Vout2}, we also obtain results when $W(x)$ is near
$-1$ or the quadratic roots of unity, along with more general expressions in
place of $\cA(x)$.

\subsection{Structure of this paper}

We structure this paper as follows. After some notation in the next subsection,
Section~2 contains the statements of our results, our general theorem followed
by two corollaries. In Section~3, we present Thue's original statement of his
Fundamentaltheorem followed by our own simplified version. In Section~4, we
establish the form of the polynomials to which Thue's Fundamentaltheorem applies.
Section~5 contains information on the roots of these polynomials. Section~6
contains two diophantine lemmas. This is followed in Section~7 by some analytic
results on the size of the numerators and denominators of the hypergeometric
polynomials as well as bounds for the values of the polynomials. Section~8
contains the proof of Theorem~\ref{thm:general-hypg}, Section~9 contains the
proof of Theorem~\ref{thm:general-hypg-unitdisk}, Finally, we prove our two
corollaries in Sections~10 and 11.

\subsection{Notation}

In order to state our results, we start with some notation. 

%
%
%
%
For positive integers $m$ and $n$ with $(m,n)=1$ and a non-negative integer $r$,
we put
\begin{displaymath}
X_{m,n,r}(x) = {} _{2}F_{1}(-r,-r-m/n;1-m/n;x),  
\end{displaymath}
where $_{2}F_{1}$ denotes the classical hypergeometric function. 

We use $X_{m,n,r}^{*}$ to denote the homogeneous polynomials derived from these
polynomials, so that  
\begin{displaymath}
X_{m,n,r}^{*}(x,y) = y^{r} X_{m,n,r}(x/y).  
\end{displaymath}

For Thue's Fundamentaltheorem itself, we will only use $X_{1,n,r}$, so for convenience
we will use $X_{n,r}$ rather than $X_{1,n,r}$ in what follows.

We let $D_{m,n,r}$ denote the smallest positive integer such that
$D_{m,n,r} X_{m,n,r}(x)$ has rational integer coefficients (and again $D_{n,r}$
in place of $D_{1,n,r}$).

For a positive integer $d$, we define $N_{d,n,r}$ to be the greatest common
divisor of the numerators of the coefficients of $X_{m,n,r}(1-dx)$.

We will use $v_{p}(x)$ to denote the largest power of a prime $p$ which divides
into the rational number $x$. With this notation, for positive integers $d$
and $n$, we put
\begin{displaymath}
\cN_{d,n} =\prod_{p|n} p^{\min(v_{p}(d), v_{p}(n)+1/(p-1))}.
\end{displaymath}

For any complex number $w$, we can write $w=se^{i \varphi}$, where $s \geq 0$ and
$-\pi < \varphi \leq \pi$ (with $\varphi=0$, if $s=0$). With such a representation,
unless otherwise stated, $w^{1/n}$ will signify $s^{1/n}e^{i \varphi/n}$ for a
positive integer $n$, where $s^{1/n}$ is the unique non-negative $n$-th root of $s$.

Lastly, following the function name in PARI, we define $\core(n)$ to be the
unique squarefree divisor, $n_{1}$, of $n$ such that $n/n_{1}$ is a perfect square.

\section{Results}

\begin{thm}
\label{thm:general-hypg}
Let $\bK$ be either $\bQ$ or an imaginary quadratic field and let $\beta_{1}$
be an algebraic integer with $[\bK(\beta_{1}):\bK] \leq 2$.

If $\bK=\bQ$ or $\bK(\beta_{1})=\bK$, then let $\tau=1$, else let $\tau$ be an
algebraic integer in $\bK$ such that $\bK(\beta_{1})=\bK(\sqrt{\tau})$.

If $\beta_{1} \in \bK$, then let $\beta_{2}$, $\gamma_{1}, \gamma_{2} \in \bK$
with the $\gamma_{i}$'s non-zero, $\beta_{2} \neq \beta_{1}$ and $\beta_{2}$
an algebraic integer.

If $[\bK(\beta_{1}):\bK] = 2$, then let $\beta_{2}$ be the algebraic conjugate of
$\beta_{1}$ over $\bK$, $\gamma_{1} \in \bK(\beta_{1})$ and $\gamma_{2}$ be the
algebraic conjugate of $\gamma_{1}$ over $\bK$ $($so $\gamma_{1}=\gamma_{2}$
if they are elements of $\bK)$.

For an algebraic integer $x \in \bK$ and a rational integer $n \geq 3$, put
$$
U(x) = -\gamma_{2} \left( x-\beta_{2} \right)^{n},
\hspace{5.0mm}
Z(x) = \gamma_{1} \left( x-\beta_{1} \right)^{n},
\hspace{5.0mm}
W(x) = \frac{Z(x)}{U(x)}
$$
and
$$
\cA(x) =
\frac{\beta_{1}\left( x - \beta_{2} \right)W(x)^{1/n}-\beta_{2}\left( x - \beta_{1} \right)}
{\left( x - \beta_{2} \right)W(x)^{1/n}-\left( x - \beta_{1} \right)}.
$$

Let $g$ be an algebraic number such that $U(x)/g$ and $Z(x)/g$ are algebraic
integers $($not necessarily in $\bK(\beta_{1}))$. For each non-negative integer
$r$, let $h_{r}$ be a non-zero algebraic integer with $h_{r}/g^{r} \in \bK$
and $|h_{r}| \leq h$ for some fixed positive real number $h$.

Let $d$ be the largest positive rational integer such that $(U(x)-Z(x))/(dg)$ is
an algebraic integer and let $\cC_{n}$ and $\cD_{n}$ be positive real numbers
such that
\begin{equation}
\label{eq:num-denom-bnd-a}
\max \left( 1, \frac{\Gamma(1-1/n) \, r!}{\Gamma(r+1-1/n)},
\frac{n\Gamma(r+1+1/n)}{\Gamma(1/n)r!} \right)
\frac{D_{n,r}}{N_{d,n,r}} < \cC_{n} \left( \frac{\cD_{n}}{\cN_{d,n}} \right)^{r}
\end{equation}
holds for all non-negative integers $r$.

Put
\begin{eqnarray*}
E      & = & \frac{|g|\cN_{d,n}}{\cD_{n}} \left\{ \min \left( \left| \sqrt{U(x)}-\sqrt{Z(x)} \right| , \left| \sqrt{U(x)}+\sqrt{Z(x)} \right| \right) \right\}^{-2}, \\
Q      & = & \frac{\cD_{n}}{|g|\cN_{d,n}} \left\{ \max \left( \left| \sqrt{U(x)}-\sqrt{Z(x)} \right| , \left| \sqrt{U(x)}+\sqrt{Z(x)} \right| \right) \right\}^{2}, \\
\kappa & = & \frac{\log Q}{\log E} \mbox{ and } \\        
c      & = & 4h |\sqrt{\tau}| \left( |x-\beta_{1}| + |x-\beta_{2}| \right) \cC_{n} Q \\
       &   & \times \max \left( 1, 5h |\sqrt{\tau}| \left| 1- W(x)^{1/n} \right| \left| x-\beta_{2} \right|
                          \left| \cA(x)-\beta_{1} \right| \cC_{n}E \right)^{\kappa}.
\end{eqnarray*}

If $E > 1$ and either $0 < W(x) < 1$ or $|W(x)|=1$ with $W(x) \neq -1$, then 
\begin{equation}
\label{eq:gen-result}
\left| \cA(x) - p/q \right| > \frac{1}{c |q|^{\kappa+1}} 
\end{equation}
for all algebraic integers $p$ and $q$ in $\bK$ with $q \neq 0$. 
\end{thm}

\begin{rem}
As we will see in the proof of Corollary~\ref{cor:cor-2}, the inclusion of the
$h_{r}$'s here can permit the use of a larger value of $g$ and hence improved
reduced values of $\kappa$.
\end{rem}

\begin{rem}
The inequality (\ref{eq:num-denom-bnd-a}) does not impose any constraint for,
as we will demonstrate in Lemma~\ref{lem:denom}, such an inequality always holds.
\end{rem}

We can also obtain a similar, though slightly weaker, result for other values
of $W(x)$ near $1$. This allows us to extend and refine the results of
Heuberger \cite{Heu}.

\begin{thm}
\label{thm:general-hypg-unitdisk}
Let $\bK$ be an imaginary quadratic field and $\beta_{1}, \beta_{2}, \gamma_{1},
\gamma_{2}$, $\tau$, $x$, $n$, $U(x)$, $Z(x)$, $W(x), \cA(x), d, g, h_{r}, h$,
$\cC_{n}, \cD_{n}, \cN_{d,n}$ be as in Theorem~$\ref{thm:general-hypg}$.

Put
\begin{eqnarray*}
E      & = & \frac{|g|\cN_{d,n}}{\cD_{n}}
             \frac{4(|U(x)|-|Z(x)-U(x)|)}{|Z(x)-U(x)|^{2}}, \\
Q      & = & \frac{\cD_{n}}{|g|\cN_{d,n}}
             2\left( \left| U(x) \right| + \left| Z(x) \right| \right), \\
\kappa & = & \frac{\log Q}{\log E} \mbox{ and } \\        
c      & = & 4h |\sqrt{\tau}| \left( |x-\beta_{1}| + |x-\beta_{2}| \right) \cC_{n} Q \\
       &   & \times \max \left( 1, 2h |\sqrt{\tau}| \left| 1-W(x)^{1/n} \right| \left| x-\beta_{2} \right| |\cA(x)-\beta_{1}| \cC_{n}E \right)^{\kappa}.
\end{eqnarray*}

If $E > 1$, $\max \left( |1-W(x)|, |1-1/W(x)| \right)<1$, then 
\begin{equation}
\label{eq:gen-unitdisk-result}
\left| \cA(x) - p/q \right| > \frac{1}{c |q|^{\kappa+1}} 
\end{equation}
for all algebraic integers $p$ and $q$ in $\bK$ with $q \neq 0$. 
\end{thm}

\begin{rem}
The condition that $\bK$ be an imaginary quadratic field is no restriction since
the case of $\bK=\bQ$ is completely covered by Theorem~\ref{thm:general-hypg}.
\end{rem}

We now give two corollaries of Theorem~\ref{thm:general-hypg}, showing how it
contains, and extends, currently-known results as well as providing new results.
They cover all cases where $[\bK(\beta_{1}):\bQ] \leq 2$.

In the first corollary, we establish effective irrationality measures for
numbers of the form $z^{1/n}$. Together with Lemma~\ref{lem:transform}, it also
strengthens Theorem~2.1 in \cite{PZ} and extends it to any algebraic number in
an imaginary quadratic field which lies on the unit circle.

\begin{cor}
\label{cor:cor-1}
Let $\bK$ be an imaginary quadratic field and $n \geq 3$, a rational integer.
Let $a$ and $b$ be algebraic integers in $\bK$ with the ideal $(a,b)=\cO_{\bK}$
and either $a/b>1$ a rational number or $|a/b|=1$ with $a/b \neq -1$. Let $d$
be the largest positive rational integer such that $(a-b)/d$ is an algebraic
integer. Let $\cC_{n}$, $\cD_{n}$ and $\cN_{d,n}$ be as in
Theorem~$\ref{thm:general-hypg}$.

Put
\begin{eqnarray*}
E      & = & \frac{\cN_{d,n}}{\cD_{n}} \left\{ \min \left( \left| \sqrt{a}-\sqrt{b} \right|, \left| \sqrt{a}+\sqrt{b} \right| \right) \right\}^{-2}, \\
Q      & = & \frac{\cD_{n}}{\cN_{d,n}} \left\{ \max \left( \left| \sqrt{a}-\sqrt{b} \right|, \left| \sqrt{a}+\sqrt{b} \right| \right) \right\}^{2}, \\
\kappa & = & \frac{\log Q}{\log E} \hspace{3.0mm} \mbox{ and } \\        
c      & = & 4|a|\cC_{n}Q \left( 2.5\left| \frac{a(a-b)}{b} \right| \cC_{n}E \right)^{\kappa}.
\end{eqnarray*}

If $E > 1$, then 
\begin{equation}
\label{eq:cor-1-result}
\left| (a/b)^{1/n} - p/q \right| > \frac{1}{c |q|^{\kappa+1}} 
\end{equation}
for all algebraic integers $p$ and $q$ in $\bK$ with $q \neq 0$. 
\end{cor}

Our second corollary covers the cases when $\beta_{1}$ and $\beta_{2}$ lie in a
quadratic extension of $\bQ$. There is some overlap with Corollary~\ref{cor:cor-1},
as we allow $\beta_{1} \in \bQ$ here, but the formulation here allows
Corollary~\ref{cor:cor-2} to be readily applied to parametrised families of
algebraic numbers.

\begin{cor}
\label{cor:cor-2}
Let $n$, $t$ and $x$ be rational integers with $n \geq 3$ and $t \neq 0$.
Let $\beta_{1}=a+b\sqrt{t}$ be an algebraic integer with $a, b \in \bQ$
and $b \neq 0$ and let $\beta_{2}=a-b\sqrt{t}$.

Let $\gamma_{1}$ be an algebraic integer in $\bQ(\sqrt{t})$ with $\gamma_{2}$
as its algebraic conjugate.

We can write $U(x) = -\gamma_{2} \left( x -\beta_{2} \right)^{n} = (u_{1} + u_{2} \sqrt{t})/2$ where $u_{1}, u_{2} \in \bZ$.
Put
\begin{eqnarray*}
g_{1}  & = & \gcd \left( u_{1}, u_{2} \right), \\
g_{2}  & = & \gcd(u_{1}/g_{1}, t), \\
g_{3}  & = & \left\{ 	\begin{array}{ll}
             1 & \mbox{if $t \equiv 1 \bmod 4$ and $(u_{1}-u_{2})/g_{1} \equiv 0 \bmod 2$}, \\
             2 & \mbox{if $t \equiv 3 \bmod 4$ and $(u_{1}-u_{2})/g_{1} \equiv 0 \bmod 2$},\\
             4 & \mbox{otherwise,}
             \end{array}
                       \right. \\
g_{4}  & = & \gcd \left( \core(g_{2}g_{3}), \frac{\gcd(2,n)n}{\gcd(u_{1}/g_{1}, \gcd(2,n)n)} \right), \\
g      & = & \frac{g_{1} \sqrt{g_{2}}}{\sqrt{g_{3}g_{4}}}, \\
E      & = & \frac{|g|\cN_{d,n}}{\cD_{n}\min \left( \left| u_{2}\sqrt{t} \pm \sqrt{u_{2}^{2}t-u_{1}^{2}} \right| \right)}, \\
Q      & = & \frac{\cD_{n}\max \left( \left| u_{2}\sqrt{t} \pm \sqrt{u_{2}^{2}t-u_{1}^{2}} \right| \right)}{|g|\cN_{d,n}}, \\
\kappa & = & \frac{\log Q}{\log E} \mbox{ and } \\        
c      & = & 4 \sqrt{|2t|} \left( |x-\beta_{1}| + |x-\beta_{2}| \right) \cC_{n} Q \\
       &   & \times \left( \max \left( 1, 5 \sqrt{|2t|} \left| 1- W(x)^{1/n} \right|
             |x-\beta_{2}| |\cA(x)-\beta_{1}| \cC_{n}E \right) \right)^{\kappa},
\end{eqnarray*}
where $d$ is the largest positive rational integer such that $u_{1}/(dg)$
is an algebraic integer and $\cA(x)$, $\cC_{n}$, $\cD_{n}$, $\cN_{d,n}$ and
$W(x)$ are as in Theorem~$\ref{thm:general-hypg}$.

If $E > 1$ and either $0 < W(x) < 1$ or $|W(x)|=1$ with $W(x) \neq -1$, then
\begin{equation}
\label{eq:cor-2-result}
\left| \cA(x) - p/q \right| > \frac{1}{c |q|^{\kappa+1}} 
\end{equation}
for all rational integers $p$ and $q$ with $q \neq 0$. 
\end{cor}

\begin{rem}
The factor $g_{4}$ here may appear wasteful as $(u/g_{1})\sqrt{g_{3}/g_{2}}$
is already an algebraic integer. It arises from an interdependence between $d$ and $g$
here. The factor of $\sqrt{g_{4}}$ allows us to increase the size of $d$ by a factor
of $g_{4}$ and hence obtain a net benefit of $\sqrt{g_{4}}$. This can be important in
practice (for example, filling the gap $1200 < t<40,000$ in \cite{ATW}).
\end{rem}

\section{Thue's Fundamentaltheorem}

\begin{lem}[Thue \cite{Thue4}] 
\label{lem:thue}
Let $F(x)$ be a polynomial of degree $n \geq 2$ and assume 
that there is a quadratic polynomial $G(x)$ with non-zero 
discriminant such that 
\begin{equation}
\label{eq:thue-diff-eqn}
G(x) \frac{{\rm d}^{2}}{{\rm d}x^{2}} \left( F(x) \right)
- (n-1)\frac{{\rm d}}{{\rm d}x} \left( G(x) \right) \frac{{\rm d}}{{\rm d}x} \left( F(x) \right)
+ \frac{n(n-1)}{2} \frac{{\rm d}^{2}}{{\rm d}x^{2}} \left( G(x) \right) F(x) = 0. 
\end{equation}

We write 
\begin{eqnarray*}
Y(x) & = & 2G(x)\frac{{\rm d}}{{\rm d}x} \left( F(x) \right)
-n\frac{{\rm d}}{{\rm d}x} \left( G(x) \right)F(x), \\  
h    & = & \frac{n^{2}-1}{4} \left( \left( \frac{{\rm d}}{{\rm d}x} \left( G(x) \right) \right)^{2}
- 2G(x) \frac{{\rm d^{2}}}{{\rm d}x^{2}} \left( G(x) \right) \right) 
\hspace{1mm} \mbox{ and } \hspace{1mm} 
\lambda = \frac{h}{n^{2}-1}. 
\end{eqnarray*}

Let us define two families of polynomials $P_{r}'(x)$
and $Q_{r}'(x)$ by the initial conditions
\begin{eqnarray*}
Q_{0}'(x) & = & \frac{2h}{3}, \\
Q_{1}'(x) & = & \frac{2(n+1)}{3} 
	       \left( G(x)\frac{{\rm d}}{{\rm d}x} \left( F(x) \right)
	              - \frac{n-1}{2}\frac{{\rm d}}{{\rm d}x} \left( G(x) \right) F(x) \right), \\ 
P_{0}'(x) & = & \frac{2hx}{3},  \\ 
P_{1}'(x) & = & xQ_{1}'(x) - \frac{2(n+1)G(x)F(x)}{3},
\end{eqnarray*}
and, for $r \geq 1$, by the recurrence equations 
\begin{eqnarray*}
\lambda (n(r+1)-1) Q_{r+1}'(x) 
& = & \left( r + \frac{1}{2} \right) Y(x) Q_{r}'(x) - (nr+1) F^{2}(x) Q_{r-1}'(x), \nonumber \\
\lambda(n(r+1)-1)P_{r+1}'(x) 
& = & \left( r + \frac{1}{2} \right) Y(x) P_{r}'(x) - (nr+1)F^{2}(x)P_{r-1}'(x).
\end{eqnarray*}

{\rm (a)} For any root $\alpha$ of $F(x)$, 
\begin{displaymath}
\alpha Q_{r}'(x) - P_{r}'(x) = S_{r}'(x),
\end{displaymath}
where $S_{r}'(x)$ is a polynomial divisible by $(x-\alpha)^{2r+1}$.

{\rm (b)} Put 
\begin{displaymath}
Z(x) = \frac{1}{2} \left( \frac{Y(x)}{2n \sqrt{\lambda}} + F(x) \right)
\hspace{2.0mm} \mbox{ and } \hspace{2.0mm}
U(x) = \frac{1}{2} \left( \frac{Y(x)}{2n \sqrt{\lambda}} - F(x) \right)  .
\end{displaymath}

Then 
\begin{eqnarray*}
(\sqrt{\lambda})^{r}Q_{r}'(x) & = & A(x)X_{n,r}^{*}(Z(x), U(x)) - B(x)X_{n,r}^{*}(U(x), Z(x)) 
\mbox{ and } \\
(\sqrt{\lambda})^{r}P_{r}'(x) & = & C(x)X_{n,r}^{*}(Z(x), U(x)) - D(x)X_{n,r}^{*}(U(x), Z(x)), 
\end{eqnarray*}
where
\begin{eqnarray*}
A(x) & = & \left( \frac{(n-1)\sqrt{\lambda}}{2F(x)} \right) Q_{1}'(x)  
	   - \left( \frac{Y(x)}{4\sqrt{\lambda}F(x)} - \frac{1}{2} \right) Q_{0}'(x), \\ 
B(x) & = & \left( \frac{(n-1)\sqrt{\lambda}}{2F(x)} \right) Q_{1}'(x)
	   - \left( \frac{Y(x)}{4\sqrt{\lambda}F(x)} + \frac{1}{2} \right) Q_{0}'(x), \\
C(x) & = & \left( \frac{(n-1)\sqrt{\lambda}}{2F(x)} \right) P_{1}'(x)
	   - \left( \frac{Y(x)}{4\sqrt{\lambda}F(x)} - \frac{1}{2} \right) P_{0}'(x) \mbox{ and } \\
D(x) & = & \left( \frac{(n-1)\sqrt{\lambda}}{2F(x)} \right) P_{1}'(x)
	   - \left( \frac{Y(x)}{4\sqrt{\lambda}F(x)} + \frac{1}{2} \right) P_{0}'(x).
\end{eqnarray*}
\end{lem}

These results can be found in Thue \cite[Theorem and equations~35--47]{Thue4}
or Chudnovsky \cite{Chud} (see, in particular, Lemma~7.1 and the remarks that
follow (pages 364--366)). 

We have added two extra hypotheses, requiring that the degree  of $F(x)$ be at
least two and that the discriminant of $G(x)$ be non-zero. If $n=1$, then $h=n-1=0$,
with the result that $A(x)=B(x)=C(x)=D(x)=0$ and the relationship between the
$P_{r}'(x)$'s and $Q_{r}'(x)$'s and the hypergeometric functions fails. When the
discriminant of $G(x)$ is zero, the recurrence relationship for the $P_{r}'(x)$'s
and $Q_{r}'(x)$'s does not hold.

Also notice that there are some differences in notation between the 
lemma above, which is similar to Chudnovsky's \cite{Chud}, and that 
of Thue. In particular, here,\\
$\bullet$ Thue's $U$ is replaced by $G$ here, \\
$\bullet$ our $n$ and $r$ are switched from \cite{Thue4}, \\
$\bullet$ our $P_{r}'(x)$ is $2(r-1)B_{n}(x)/3$ and our $Q_{r}'(x)$ is $2(r-1)A_{n}(x)/3$
in Thue's notation (we use the superscript as we will simplify these polynomials
further in what follows), \\
$\bullet$ we capitalise Thue's $a$, $b$, $c$, $d$, and $z$ \\
$\bullet$ what we call $Y(x)$ and $U(x)$ respectively, correspond to $2H(x)$
and $y(x)$ respectively in Thue's paper \\
$\bullet$ we label Thue's $U_{n}(z,y)$ as $X_{n,r}^{*}(Z(x), U(x))$.

However, this lemma can be simplified considerably and that is the objective of this section.

We start with a result regarding the differential equation in (\ref{eq:thue-diff-eqn}).

\begin{lem}
\label{lem:diff-eqn}
Let $m$ and $n$ be positive integers with $n \geq m$ and let $\beta_{1}, \ldots, \beta_{m}$ be distinct
complex numbers. Put $G(x)=(x-\beta_{1}) \cdots (x-\beta_{m})$.

An analytic function $F(x)$ is a solution of the differential equation
\begin{equation}
\label{eq:gen-diff-eqn}
\sum_{i=0}^{m} (-1)^{i} {n-m+i \choose i}
\frac{{\rm d}^{i}}{{\rm d}x^{i}} \left( G(x) \right) \frac{{\rm d}^{m-i}}{{\rm d}x^{m-i}} \left( F(x) \right) = 0,
\end{equation}
if and only if it is of the form
\begin{displaymath}
F(x) = \sum_{i=1}^{m} \gamma_{i} \left( x - \beta_{i} \right)^{n},
\end{displaymath}
for some choice of $\gamma_{1}, \ldots, \gamma_{m} \in \bC$.
\end{lem}

\begin{proof}
Note that (\ref{eq:gen-diff-eqn}) is a homogeneous linear differential equation
of order $m$. The theory of these equations is well-understood (see, for example,
Chapter~4 of \cite{BD}).

By Theorem~4.1.2 of \cite{BD}, given $m$ linearly independent solutions
$(F_{1}(x)$, \ldots, $F_{m}(x))$
of the differential equation, then any solution is given by
$\gamma_{1}F_{1}(x)$ $+\cdots+$ $\gamma_{m}F_{m}(x)$ for some constants
$\gamma_{1}, \ldots, \gamma_{m}$. Here we show that $F_{1}(x)=(x-\beta_{1})^{n}, \ldots ,
F_{m}(x)=(x-\beta_{m})^{n}$ are such linearly independent solutions.

Putting $F(x)=F_{j}(x)$,
$$
\frac{{\rm d}^{m-i}}{{\rm d}x^{m-i}} \left( F(x) \right)
= \frac{n!}{(n-(m-i))!} \left( x-\beta_{j} \right)^{n-(m-i)},
$$
so we can write (\ref{eq:gen-diff-eqn}) as
\begin{eqnarray*}
&   & \sum_{i=0}^{m} (-1)^{i} {n-m+i \choose i} \frac{{\rm d}^{i}}{{\rm d}x^{i}} \left(G(x) \right) \frac{n!}{(n-(m-i))!} \left( x - \beta_{j} \right)^{n-(m-i)} \\
& = & \frac{n!}{(n-m)!} \left( x - \beta_{j} \right)^{n-m} \sum_{i=0}^{m} \frac{(-1)^{i}}{i!} \frac{{\rm d}^{i}}{{\rm d}x^{i}} \left(G(x) \right) \left( x - \beta_{j} \right)^{i} \\
\end{eqnarray*}

Note that the sum in the last expression is in fact the Taylor series expansion of $G(\beta_{j})=0$,
since $\deg G(x)=m$. Therefore, the entire expression is $0$. Hence $\left( x-\beta_{j} \right)^{n}$
satisfies the required differential equation for each $j=1,\ldots,m$ and it only remains to show that
these $m$ solutions are linearly independent.

This is equivalent to showing that their Wronskian is not always zero. We can
write this Wronskian as
\begin{eqnarray*}
&  & \det
\left(
\begin{array}{ccc}
(x-\beta_{1})^{n}    & \cdots & (x-\beta_{m})^{n} \\
n(x-\beta_{1})^{n-1} & \cdots & n(x-\beta_{m})^{n-1} \\
                     & \cdots & \\
\displaystyle\frac{n!(x-\beta_{1})^{n-(m-1)}}{(n-m+1)!} & \cdots &
\displaystyle\frac{n!(x-\beta_{m})^{n-(m-1)}}{(n-m+1)!}
\end{array}
\right)  \\
& = & \left( \prod_{i=1}^{m} \frac{n!(x-\beta_{i})^{n-(m-1)}}{(n-i+1)!} \right)
\det
\left(
\begin{array}{ccc}
(x-\beta_{1})^{m-1} & \cdots & (x-\beta_{m})^{m-1} \\
(x-\beta_{1})^{m-2} & \cdots & (x-\beta_{m})^{m-2} \\
                                 & \cdots & \\
1 & \cdots & 1
\end{array}
\right)  \\
& = & \left( \prod_{i=1}^{m} \frac{n!(x-\beta_{i})^{n-(m-1)}}{(n-i+1)!} \right)
\prod_{1 \leq i < j \leq m} \left( (x-\beta_{i})-(x-\beta_{j}) \right).
\end{eqnarray*}

This function is identically zero only if the $\beta_{i}$'s are not all distinct,  a condition
which we exclude here.
\end{proof}

We now present our simplified version of Lemma~\ref{lem:thue}.

\begin{lem}
\label{lem:thue-simp}
Let $\beta_{1}, \beta_{2}, \gamma_{1}$ and $\gamma_{2}$ be complex numbers with
$\beta_{1} \neq \beta_{2}$. For any integer $n \geq 2$, we put
$$
U(x) = -\gamma_{2} \left( x-\beta_{2} \right)^{n}
\hspace{3.0mm} \mbox{and} \hspace{3.0mm}
Z(x) = \gamma_{1} \left( x-\beta_{1} \right)^{n}.
$$

For all non-negative integers $r$, we define 
\begin{eqnarray*}
Q_{r}(x) & = & \left( x - \beta_{2} \right) X_{n,r}^{*}(Z(x), U(x))
- \left( x - \beta_{1} \right)X_{n,r}^{*}(U(x), Z(x)) 
\mbox{ and } \\
P_{r}(x) & = & \beta_{1} \left( x - \beta_{2} \right)X_{n,r}^{*}(Z(x), U(x))
- \beta_{2} \left( x - \beta_{1} \right) X_{n,r}^{*}(U(x), Z(x)).
\end{eqnarray*}

For any root, $\alpha$, of
\begin{displaymath}
F(x) = \gamma_{1} \left( x - \beta_{1} \right)^{n} + \gamma_{2} \left( x - \beta_{2} \right)^{n},
\end{displaymath}
the polynomial 
\begin{displaymath}
S_{r}(x) = \alpha Q_{r}(x) - P_{r}(x)
\end{displaymath}
is divisible by $(x-\alpha)^{2r+1}$. 
\end{lem}

\begin{proof}
First note that we may assume that $G(x)$ is monic since wherever $G(x)$ is used in
Lemma~\ref{lem:thue}, the leading coefficient can be eliminated. Therefore, we can write
$G(x) = \left( x-\beta_{1} \right) \left( x-\beta_{2} \right)$.

Applying Lemma~\ref{lem:diff-eqn} with $m=2$, we see that a polynomial $F(x)$ satisfies
the differential equation in (\ref{eq:thue-diff-eqn}) if and only if it is of the form above.

Also $h=(n^{2}-1)(\beta_{1}-\beta_{2})^{2}/4$ and $\lambda=(\beta_{1}-\beta_{2})^{2}/4$.

Next, we need to calculate Thue's $Y(x)$.
\begin{eqnarray*}
Y(x) & = & 2G(x)\frac{{\rm d}}{{\rm d}x} \left( F(x) \right) -n\frac{{\rm d}}{{\rm d}x} \left( G(x) \right)F(x), \\
& = & 2 \left( x-\beta_{1} \right) \left( x-\beta_{2} \right) \left( \gamma_{1}n(x-\beta_{1})^{n-1}+\gamma_{2}n(x-\beta_{2})^{n-1} \right) \\
&   & - n \left( 2x - \left( \beta_{1}+\beta_{2} \right) \right) \left( \gamma_{1}(x-\beta_{1})^{n}+\gamma_{2}(x-\beta_{2})^{n} \right) \\
& = & n(\beta_{1}-\beta_{2}) \left( \gamma_{1}(x-\beta_{1})^{n} - \gamma_{2}(x-\beta_{2})^{n} \right) \\
& = & 2n \sqrt{\lambda} \left( \gamma_{1}(x-\beta_{1})^{n} - \gamma_{2}(x-\beta_{2})^{n} \right).
\end{eqnarray*}

Thus
\begin{eqnarray*}
Z(x) & = & \frac{1}{2} \left( \frac{Y(x)}{2n \sqrt{\lambda}} + F(x) \right) \\
& = & \frac{1}{2} \left( \gamma_{1}(x-\beta_{1})^{n} - \gamma_{2}(x-\beta_{2})^{n}
+ \gamma_{1} (x - \beta_{1})^{n} + \gamma_{2} (x - \beta_{2})^{n} \right) \\
& = & \gamma_{1} (x - \beta_{1})^{n}.
\end{eqnarray*}

Similarly, we find that $U(x)=-\gamma_{2}(x-\beta_{2})^{n}$.

Now we determine the expressions for $A(x)$, $B(x)$, $C(x)$ and $D(x)$.
\begin{eqnarray*}
A(x) & = & \frac{2(n-1)\lambda Q_{1}(x) - Y(x)Q_{0}(x) + 2\sqrt{\lambda}F(x)Q_{0}(x)}{4\sqrt{\lambda} F(x)} \\
& = & \frac{h}{6\sqrt{\lambda} F(x)} \left( 2G(x) \frac{{\rm d}}{{\rm d}x} \left( F(x) \right)
- (n-1)\frac{{\rm d}}{{\rm d}x} \left( G(x) \right) F(x) - Y(x) + 2\sqrt{\lambda}F(x) \right) \\
& = & \frac{h}{6\sqrt{\lambda} F(x)} \left( \frac{{\rm d}}{{\rm d}x} \left( G(x) \right) F(x) + 2\sqrt{\lambda}F(x) \right) \\
& = & \frac{(n^{2}-1)\sqrt{\lambda}}{6} \left( \frac{{\rm d}}{{\rm d}x} \left( G(x) \right) + 2\sqrt{\lambda} \right) \\
& = & \frac{n^{2}-1}{6} \left( \beta_{1}-\beta_{2} \right) \left( x - \beta_{2} \right).
\end{eqnarray*}

A similar argument establishes that
$$
B(x) = \frac{n^{2}-1}{6} \left( \beta_{1}-\beta_{2} \right) \left( x - \beta_{1} \right),
$$
as well as the relationships 
$C(x)=\beta_{1}A(x)$ and $D(x)=\beta_{2}B(x)$.

Therefore,
\begin{eqnarray*}
(\sqrt{\lambda})^{r}Q_{r}'(x) & = & A(x)X_{n,r}^{*}(Z(x), U(x)) - B(x)X_{n,r}^{*}(U(x), Z(x)) \\
& = & \frac{n^{2}-1}{6} \left( \beta_{1}-\beta_{2} \right) \\
&   & \times \left\{ \left( x - \beta_{2} \right) X_{n,r}^{*}(Z(x), U(x))
- \left( x - \beta_{1} \right) X_{n,r}^{*}(U(x), Z(x)) \right\} \\
\mbox{ and } \\
(\sqrt{\lambda})^{r}P_{r}'(x) & = & C(x)X_{n,r}^{*}(Z(x), U(x)) - D(x)X_{n,r}^{*}(U(x), Z(x)) \\
& = & \frac{n^{2}-1}{6} \left( \beta_{1}-\beta_{2} \right) \\
&   & \times \left\{ \beta_{1} \left( x - \beta_{2} \right) X_{n,r}^{*}(Z(x), U(x))
- \beta_{2} \left( x - \beta_{1} \right) X_{n,r}^{*}(U(x), Z(x)) \right\}.
\end{eqnarray*}

We now set $P_{r}(x)$, $Q_{r}(x)$ and $S_{r}(x)$ to be
$6/((n^{2}-1)(\beta_{1}-\beta_{2}))$ times $(\sqrt{\lambda})^{r}P_{r}'(x)$,
$(\sqrt{\lambda})^{r}Q_{r}'(x)$ and $(\sqrt{\lambda})^{r}S_{r}'(x)$, respectively.

From the statement of Thue's Fundamentaltheorem (Lemma~\ref{lem:thue}~(a)),
for any root $\alpha$ of $F(x)$,
\begin{displaymath}
\alpha Q_{r}(x) - P_{r}(x) = S_{r}(x)
\end{displaymath}
where $S_{r}(x)$ is a polynomial divisible by $(x - \alpha)^{2r+1}$.
\end{proof}

\section{The Form of The Polynomials}

\begin{lem}
\label{lem:poly-form}
Let $\beta_{1}, \beta_{2}, \gamma_{1}$ and $\gamma_{2}$ be complex numbers
with $\beta_{1} \neq \beta_{2}$ and let $n$ be an integer with $n \geq 3$.
For any number field $\bK$, we have
\begin{displaymath}
0 \neq F(x) = \gamma_{1} \left( x - \beta_{1} \right)^{n} + \gamma_{2} \left( x - \beta_{2} \right)^{n}
\in \bK[x],
\end{displaymath}
if and only if either

{\rm (a)} one of the $\gamma_{i}$'s is zero $($say $\gamma_{1})$, $\beta_{1}$
is any complex number, $\gamma_{2}$ is a non-zero element of $\bK$ and
$\beta_{2}$ is element of $\bK$ other than $\beta_{1}$,

{\rm (b)} $\beta_{1}, \beta_{2}, \gamma_{1}, \gamma_{2} \in \bK$,
or

{\rm (c)} $[\bK(\beta_{1}):\bK]=2$ and $\beta_{2}$ is the algebraic conjugate
of $\beta_{1}$ over $\bK$, $\gamma_{1} \in \bK(\beta_{1})$ and $\gamma_{2}$ is
the algebraic conjugate of $\gamma_{1}$ over $\bK$ $($so $\gamma_{1}=\gamma_{2}$
if they are elements of $\bK)$.
\end{lem}

\begin{rem}
The condition $n \geq 3$ here is necessary. If $\beta_{1}=\pi$,
$\beta_{2}=-1/\pi$, $\gamma_{1}=1/(\pi^{2}+1)$ and $\gamma_{2}=\pi^{2}/(\pi^{2}+1)$
with $n=2$, then $F(x)=x^{2}+1$.

Here we have transcendental values for $\beta_{1}$, $\beta_{2}$, $\gamma_{1}$ and
$\gamma_{2}$, yet $F(x) \in \bQ[x]$.
\end{rem}

\begin{proof}
We will consider the four highest order coefficients of $F(x)$:
\begin{eqnarray}
\label{eq:poly-coeff}
\gamma_{1} + \gamma_{2} & = & a_{1} \in \bK, \nonumber \\
\gamma_{1}\beta_{1} + \gamma_{2}\beta_{2} & = & a_{2} \in \bK, \\
\gamma_{1}\beta_{1}^{2} + \gamma_{2}\beta_{2}^{2} & = & a_{3} \in \bK, \nonumber \\
\gamma_{1}\beta_{1}^{3} + \gamma_{2}\beta_{2}^{3} & = & a_{4} \in \bK. \nonumber
\end{eqnarray}

Using these expressions, we find that
\begin{eqnarray*}
\frac{a_{3}^{2} -a_{2}a_{3}\beta_{2}+(a_{2}^{2}-a_{1}a_{3})\beta_{2}^{2}}
{a_{2}-a_{1}\beta_{2}}
& = & \frac{\beta_{1}^{4}\gamma_{1}^{2} - \beta_{1}^{3}\beta_{2}\gamma_{1}^{2}
+ \beta_{1}\beta_{2}^{3}\gamma_{1}\gamma_{2} - \beta_{2}^{4}\gamma_{1}\gamma_{2}}
{\gamma_{1} \left( \beta_{1} - \beta_{2} \right)} \\
& = & \gamma_{1}\beta_{1}^{3} + \gamma_{2}\beta_{2}^{3}.
\end{eqnarray*}

If $a_{2}-a_{1}\beta_{2} = \gamma_{1} \left( \beta_{1} - \beta_{2} \right) =0$, then
$\gamma_{1}=0$ (since we assumed $\beta_{1} \neq \beta_{2}$). From the
expression for our polynomial, this implies that $\beta_{1}$ can be any complex
number and that $\gamma_{2}$ must be an element of $\bK$. If $\gamma_{2}=0$, then
$\beta_{2}$ can be any complex number ($\neq \beta_{1}$). And if $\gamma_{2} \neq 0$,
then $\beta_{2}$ must be an element of $\bK$ (again, $\neq \beta_{1}$).

These cases constitute part~(a) of the lemma, along with the assumption that
$F(x) \neq 0$, so we can assume $a_{2}-a_{1}\beta_{2} \neq 0$ in the remainder
of the proof.

From the first and last terms of the above relationship, we obtain a polynomial,
$f(x)$ such that $f(\beta_{2})=0$. Namely,
\begin{equation}
\label{eq:beta2-poly}
f(\beta_{2}) = (a_{2}^{2}-a_{1}a_{3})\beta_{2}^{2} + (a_{1}a_{4}-a_{2}a_{3})\beta_{2} + (a_{3}^{2}-a_{2}a_{4}) = 0.
\end{equation}

Therefore, $\beta_{2}$ is an algebraic number of degree at most $2$ over $\bK$.

From the expression in (\ref{eq:poly-coeff}) for the $a_{i}$'s, we find that
\begin{equation}
\label{eq:c1-exp}
\gamma_{1} = \frac{a_{2}-a_{1}\beta_{2}}{\beta_{1}-\beta_{2}},
\end{equation}
\begin{equation}
\label{eq:c2-exp}
\gamma_{2} = \frac{a_{1}(\beta_{1}-\beta_{2})-(a_{2}-a_{1}\beta_{2})}{\beta_{1}-\beta_{2}}
= \frac{a_{1}\beta_{1}-a_{2}}{\beta_{1}-\beta_{2}}
\end{equation}
and
\begin{equation}
\label{eq:alpha1-exp}
\beta_{1} = \frac{a_{3}-a_{2}\beta_{2}}{a_{2}-a_{1}\beta_{2}}.
\end{equation}

Let us consider the case of $\beta_{2} \in \bK$. From the expressions above,
we see that $\beta_{1}, \gamma_{1}, \gamma_{2} \in \bK$.
Hence we find ourselves in case~(b).

Therefore, in what follows, we assume that $\beta_{2} \not\in \bK$.

We now show that $\beta_{1}$ is the algebraic conjugate of $\beta_{2}$. To
demonstrate this, we substitute the expression for $\beta_{1}$ in (\ref{eq:alpha1-exp})
into the polynomial $f(x)$. We find that
\begin{eqnarray*}
& & (a_{2}-a_{1}\beta_{2})^{2} f \left( \beta_{1} \right) \\
& = & \left( a_{2}^{2}-a_{1}a_{3} \right) (a_{3}-a_{2}\beta_{2})^{2}
+ (a_{1}a_{4}-a_{2}a_{3})(a_{3}-a_{2}\beta_{2})(a_{2}-a_{1}\beta_{2}) \\
& & +(a_{3}^{2}-a_{2}a_{4})(a_{2}-a_{1}\beta_{2})^{2} \\
& = & \left( a_{2}^{2}-a_{1}a_{3} \right) (a_{3}-a_{2}\beta_{2})^{2}
+ (a_{2}-a_{1}\beta_{2}) \left( a_{1}a_{3}a_{4} -a_{2}^2a_{4}+a_{2}^2a_{3}\beta_{2} - a_{1}a_{3}^{2}\beta_{2} \right) \\
& = & (a_{2}^{2}-a_{1}a_{3}) \left( (a_{3}-a_{2}\beta_{2})^{2} - (a_{2}-a_{1}\beta_{2})(a_{4}-a_{3}\beta_{2}) \right) \\
& = & (a_{2}^{2}-a_{1}a_{3})f(\beta_{2}) = 0.
\end{eqnarray*}

Therefore $\beta_{1}$ is the algebraic conjugate of $\beta_{2}$.

Hence, from (\ref{eq:c1-exp}), the algebraic conjugate of $\gamma_{1}$ is
$(a_{2}-a_{1}\beta_{1})/(\beta_{2}-\beta_{1})=\gamma_{2}$, as required.
\end{proof}

\begin{rem}
From a diophantine point-of-view, there is no interest in the cases
of $\gamma_{1}=0$ or $\gamma_{2}=0$ (that is part~(a) of this lemma), since
the resulting polynomial is a power of $(x-\beta_{2})$, where $\beta_{2} \in \bK$.
So in the following we shall not consider this case any further.
\end{rem}

\section{Roots of These Polynomials}

We start with the following lemma describing the roots themselves.

\begin{lem}
\label{lem:root-exp1}
Let $n$, $\beta_{1}, \beta_{2}, \gamma_{1}, \gamma_{2}$ and $F(x)$ be as above.

Then
\begin{displaymath}
\alpha = \frac{\beta_{1} (-\gamma_{1}/\gamma_{2})^{1/n} - \beta_{2}}{(-\gamma_{1}/\gamma_{2})^{1/n} - 1}
\end{displaymath}
is a root of $F(x)$ for each $n$-th root of $-\gamma_{1}/\gamma_{2}$, except $1$ in the case of $\gamma_{1}=-\gamma_{2}$.

Furthermore, for any two distinct $n$-th roots of $-\gamma_{1}/\gamma_{2}$  $($again excluding
$1$ in the case of $\gamma_{1}=-\gamma_{2})$, the corresponding $\alpha$'s are distinct.
\end{lem}

\begin{proof}
We start by substituting the above expression for $\alpha$ into $F(x)$:
\begin{eqnarray*}
F(\alpha) & = & \gamma_{1} { \left( \frac{\beta_{1} (-\gamma_{1}/\gamma_{2})^{1/n} - \beta_{2}}{(-\gamma_{1}/\gamma_{2})^{1/n} - 1} - \beta_{1} \right) }^{n}
                         + \gamma_{2} { \left( \frac{\beta_{1} (-\gamma_{1}/\gamma_{2})^{1/n} - \beta_{2}}{(-\gamma_{1}/\gamma_{2})^{1/n} -1} - \beta_{2} \right) }^{n} \\
& = & \gamma_{1} { \left( \frac{\beta_{1}-\beta_{2}}{(-\gamma_{1}/\gamma_{2})^{1/n} - 1} \right) }^{n}
           + \gamma_{2} { \left( \frac{ \left( \beta_{1} - \beta_{2} \right) (-\gamma_{1}/\gamma_{2})^{1/n}}{(-\gamma_{1}/\gamma_{2})^{1/n} - 1} \right) }^{n} \\
& = & \frac{\left( \beta_{1} - \beta_{2} \right)^{n}  \left( \gamma_{1} + \gamma_{2} (-\gamma_{1}/\gamma_{2}) \right)}
           { { \left( (-\gamma_{1}/\gamma_{2})^{1/n} - 1 \right) }^{n}}  = 0.
\end{eqnarray*}

Next, we consider when two of these $\alpha$'s are equal. Let $(-\gamma_{1}/\gamma_{2})^{1/n}$ be a fixed $n$-th root of $-\gamma_{1}/\gamma_{2}$.
Suppose that
\begin{displaymath}
\frac{\beta_{1} (-\gamma_{1}/\gamma_{2})^{1/n} - \beta_{2}}{(-\gamma_{1}/\gamma_{2})^{1/n} - 1}
= \frac{\beta_{1} \zeta_{n}^{k}(-\gamma_{1}/\gamma_{2})^{1/n} - \beta_{2}}{\zeta_{n}^{k}(-\gamma_{1}/\gamma_{2})^{1/n} - 1},
\end{displaymath}
for some $\zeta_{n}^{k}=\exp(2\pi ik/n)$.

Then
\begin{eqnarray*}
& & \beta_{1}\zeta_{n}^{k}(-\gamma_{1}/\gamma_{2})^{2/n}
- \beta_{1}(-\gamma_{1}/\gamma_{2})^{1/n}
- \beta_{2}\zeta_{n}^{k}(-\gamma_{1}/\gamma_{2})^{1/n} + \beta_{2} \\
& = &
\beta_{1}\zeta_{n}^{k}(-\gamma_{1}/\gamma_{2})^{2/n}
- \beta_{2}(-\gamma_{1}/\gamma_{2})^{1/n}
- \beta_{1}\zeta_{n}^{k}(-\gamma_{1}/\gamma_{2})^{1/n} + \beta_{2}.
\end{eqnarray*}

So
\begin{displaymath}
\left( \beta_{1} - \beta_{2} \right) \zeta_{n}^{k} (-\gamma_{1}/\gamma_{2})^{1/n}
= 
\left( \beta_{1} - \beta_{2} \right) (-\gamma_{1}/\gamma_{2})^{1/n}.
\end{displaymath}
This implies that either $\beta_{1} = \beta_{2}$ (a condition which we exclude), $\gamma_{1}=0$ and $\gamma_{2} \neq 0$
(which we have again excluded, see the note at the end of the previous section) or $\zeta_{n}^{k}=1$,
which is to say that the two $\alpha$'s are equal.
\end{proof}

In the following lemma, we determine when the roots of the polynomials are real
for polynomials with rational coefficients.

\begin{lem}
\label{lem:root-exp2}
Let $n$, $\beta_{1}, \beta_{2}, \gamma_{1}, \gamma_{2}$ and $F(x)$ be as above.

{\rm (a)} If $\bK=\bQ$ and $\bK(\beta_{1})$ is an imaginary quadratic field,
then $F(x)$ has $n$ real roots.

{\rm (b)} Suppose that $\bK=\bQ$ and $\bK(\beta_{1})$ is contained in a real
quadratic field and write $\beta_{1}=a+b\sqrt{t}$ with $a, b \in \bQ$.

If $-\gamma_{1}/\gamma_{2}>0$, then $F(x)$ has two real roots
for $n$ even and one real root for $n$ odd. These roots are
\begin{equation}
\label{eq:rts-cpos-rootpos1}
\alpha_{1} = a+b\sqrt{t} \frac{(-\gamma_{1}/\gamma_{2})^{1/n}+1}
{(-\gamma_{1}/\gamma_{2})^{1/n}-1}
\end{equation}
and, for $n$ even,
\begin{equation}
\label{eq:rts-cpos-rootpos2}
\alpha_{2} = a+\frac{tb^{2}}{\alpha_{1}-a}
= a+b\sqrt{t} \frac{(-\gamma_{1}/\gamma_{2})^{1/n}-1}
{(-\gamma_{1}/\gamma_{2})^{1/n}+1},
\end{equation}
where $(-\gamma_{1}/\gamma_{2})^{1/n}$ denotes the unique positive real $n$-th root of $-\gamma_{1}/\gamma_{2}$.

If $-\gamma_{1}/\gamma_{2}<0$, then $F(x)$ has no real roots for $n$ even and one
real root, $\alpha_{1}$ above, for $n$ odd, where $(-\gamma_{1}/\gamma_{2})^{1/n}$
denotes the unique negative real $n$-th root of $-\gamma_{1}/\gamma_{2}$.
\end{lem}

\begin{proof}
When $\bK(\beta_{1})=\bQ$, the result is well-known, so we restrict our attention
to the case of $[\bK(\beta_{1}):\bQ]=2$. In this case, we can write $\beta_{1}=a+b\sqrt{t}$
and $\beta_{2}=a-b\sqrt{t}$, where $a, b \in \bQ$.

(a) From Lemma~\ref{lem:root-exp1}, we know that as $j$ runs through the integers
from $0$ to $n-1$,
\begin{displaymath}
\frac{\beta_{1} e^{2\pi i j/n} (-\gamma_{1}/\gamma_{2})^{1/n} - \beta_{2}}
             {e^{2\pi i j/n}(-\gamma_{1}/\gamma_{2})^{1/n} - 1}
\end{displaymath}
runs through the roots where $(-\gamma_{1}/\gamma_{2})^{1/n}$ denotes a fixed root of $-\gamma_{1}/\gamma_{2}$.

Multiplying the numerator and denominator by the complex conjugate of the
denominator and substituting the expressions for $\beta_{1}, \beta_{2}$
and $e^{2\pi i j/n}$, we find that the roots are of the form
\begin{displaymath}
a+b\sqrt{t} \frac{\left| (-\gamma_{1}/\gamma_{2})^{1/n} \right|^{2}-1
- 2i \left\{ \sin(2\pi j/n) \Re \left( (-\gamma_{1}/\gamma_{2})^{1/n} \right)
+ \cos(2\pi j/n) \Im \left( (-\gamma_{1}/\gamma_{2})^{1/n} \right) \right\} }
{\left| (-\gamma_{1}/\gamma_{2})^{1/n} \right|^{2}+1
+ 2 \sin(2\pi j/n) \Im \left( (-\gamma_{1}/\gamma_{2})^{1/n} \right)
- 2 \cos(2\pi j/n) \Re \left( (-\gamma_{1}/\gamma_{2})^{1/n} \right) }.
\end{displaymath}

If $t<0$, then $\gamma_{1}$ and $\gamma_{2}$ are also complex conjugates, and
$\left| (-\gamma_{1}/\gamma_{2})^{1/n} \right|^{2}=1$, so the roots are of the form
\begin{displaymath}
a+b\sqrt{-t} \frac{\sin(2\pi j/n) \Re \left( (-\gamma_{1}/\gamma_{2})^{1/n} \right)
+ \cos(2\pi j/n) \Im \left( (-\gamma_{1}/\gamma_{2})^{1/n} \right) }
{1 + \sin(2\pi j/n) \Im \left( (-\gamma_{1}/\gamma_{2})^{1/n} \right)
- \cos(2\pi j/n) \Re \left( (-\gamma_{1}/\gamma_{2})^{1/n} \right) }.
\end{displaymath}

So all the roots are real numbers.

(b) Now suppose that $t>0$ and $-\gamma_{1}/\gamma_{2}>0$.
Then $\Re \left( (-\gamma_{1}/\gamma_{2})^{1/n} \right) =
(-\gamma_{1}/\gamma_{2})^{1/n}$ and $\Im \left( (-\gamma_{1}/\gamma_{2})^{1/n} \right) =0$,
so we can write the roots as
\begin{displaymath}
a+b\sqrt{t} \frac{(-\gamma_{1}/\gamma_{2})^{2/n}-1
- 2i \sin(2\pi j/n) (-\gamma_{1}/\gamma_{2})^{1/n}}
{(-\gamma_{1}/\gamma_{2})^{2/n} + 1 - 2 \cos(2\pi j/n) (-\gamma_{1}/\gamma_{2})^{1/n} }.
\end{displaymath}

These roots are real if and only if their imaginary part is zero,
which only happens $2j$ is a multiple of $n$ (i.e., $j=0$ or $j=n/2$).
Hence, there are precisely two real roots when $n$ is even and
precisely one real root when $n$ is odd.

These roots are
\begin{eqnarray*}
a+\frac{b\sqrt{t} \left( (-\gamma_{1}/\gamma_{2})^{2/n}-1 \right)}
{(-\gamma_{1}/\gamma_{2})^{2/n}-2(-\gamma_{1}/\gamma_{2})^{1/n}+1}
& = & a+\frac{b\sqrt{t} \left( (-\gamma_{1}/\gamma_{2})^{2/n}-1 \right)}
{\left( (-\gamma_{1}/\gamma_{2})^{1/n}-1 \right)^{2}} \\
& = & a+\frac{b\sqrt{t} \left( (-\gamma_{1}/\gamma_{2})^{1/n}+1 \right)}
{(-\gamma_{1}/\gamma_{2})^{1/n}-1}
\end{eqnarray*}
and, similarly for $n$ even,
\begin{eqnarray*}
a+\frac{b\sqrt{t} \left( (-\gamma_{1}/\gamma_{2})^{2/n}-1 \right)}
{(-\gamma_{1}/\gamma_{2})^{2/n}+2(-\gamma_{1}/\gamma_{2})^{1/n}+1}
& = & a+\frac{b\sqrt{t} \left( (-\gamma_{1}/\gamma_{2})^{1/n}-1 \right)}
{(-\gamma_{1}/\gamma_{2})^{1/n}+1}.
\end{eqnarray*}

If $-\gamma_{1}/\gamma_{2}<0$ and $n$ is odd, then we let $(-\gamma_{1}/\gamma_{2})^{1/n}$
denote the unique negative real $n$-th root of $-\gamma_{1}/\gamma_{2}$
and by the same argument as above, there is one real root of $F(x)$ and it
is of the form
\begin{eqnarray*}
a+\frac{b\sqrt{t} \left( (-\gamma_{1}/\gamma_{2})^{2/n}-1 \right)}
{(-\gamma_{1}/\gamma_{2})^{2/n}-2(-\gamma_{1}/\gamma_{2})^{1/n}+1}
& = & a+\frac{b\sqrt{t} \left( (-\gamma_{1}/\gamma_{2})^{1/n}+1 \right)}
{\left( (-\gamma_{1}/\gamma_{2})^{1/n}-1 \right)}.
\end{eqnarray*}

If $-\gamma_{1}/\gamma_{2}<0$ and $n$ is even, then the roots are as above and
can be real only if
\begin{equation}
\label{eq:zero-cond}
\sin(2\pi j/n) \Re \left( (-\gamma_{1}/\gamma_{2})^{1/n} \right)
+ \cos(2\pi j/n) \Im \left( (-\gamma_{1}/\gamma_{2})^{1/n} \right)
\end{equation}
is zero. For $n>2$, both the real and imaginary parts of $(-\gamma_{1}/\gamma_{2})^{1/n}$
are non-zero, which means that for (\ref{eq:zero-cond}) to be zero, both
$\cos(2\pi j/n)$ and $\sin(2\pi j/n)$ must be $0$. This is impossible, hence
there are no real roots in this case.
\end{proof}

\begin{lem}
\label{lem:beta-exp}
Let $\cA(x)$ and $W(x)$ be as in Theorem~$\ref{thm:general-hypg}$ and let
$F(x)$ be as above. For any $x \in \bC$ such that $W(x)$ is not a negative
real number or zero, $F(\cA(x))=0$.

Furthermore, for each root, $\alpha$, of $F(x)$, we can find a value of $x$
such that $\cA(x)=\alpha$ $($in particular, $\cA(\alpha)=\alpha)$.
\end{lem}

\begin{proof}
We can write $W(x)^{1/n}$ as $e^{2\pi i k/n}\left( x - \beta_{1} \right) (-\gamma_{1}/\gamma_{2})^{1/n}
/\left( x - \beta_{2} \right)$ for some integer $k$. Hence
$$
\cA(x)
= \frac{e^{2\pi i k/n}\beta_{1} (-\gamma_{1}/\gamma_{2})^{1/n} - \beta_{2}}
       {e^{2\pi i k/n}(-\gamma_{1}/\gamma_{2})^{1/n} - 1}.
$$
By Lemma~\ref{lem:root-exp1}, this quantity is a root of $F(x)$.

To show that $\cA(\alpha)=\alpha$, observe that since $\alpha$ is a root of
$F(x)$, we have $\gamma_{1} \left( \alpha-\beta_{1} \right)^{n}
+\gamma_{2} \left( \alpha-\beta_{2} \right)^{n}=0$
and hence
$$
W(\alpha) = -\frac{\gamma_{1} \left( \alpha-\beta_{1} \right)^{n}}{\gamma_{2} \left( \alpha-\beta_{2} \right)^{n}}=1.
$$

Therefore, by our choice of $n$-th root,
\begin{displaymath}
\cA(\alpha) = \frac{\beta_{1} (\alpha-\beta_{2}) W(\alpha)^{1/n} - \beta_{2} (\alpha-\beta_{1})}
                        {(\alpha-\beta_{2}) W(\alpha)^{1/n} - (\alpha-\beta_{1})}
= \frac{\beta_{1} (\alpha-\beta_{2}) - \beta_{2} (\alpha-\beta_{1})}
           {(\alpha-\beta_{2}) - (\alpha-\beta_{1})} = \alpha. 
\end{displaymath}
\end{proof}

\section{Diophantine Lemmas}

The following lemma is used to obtain an effective approximation measure for
a complex number $\theta$ from a sequence of ``good'' approximations in an
imaginary quadratic field.

\begin{lem}
\label{lem:approx}
Let $\theta \in \bC$ and let $\bK$ be either $\bQ$ or an imaginary quadratic field.
Suppose that there exist real numbers $k_{0},l_{0} > 0$ and $E,Q > 1$ such that
for all non-negative integers $r$, there are algebraic integers $p_{r}$ and $q_{r}$
in $\bK$ with $|q_{r}|<k_{0}Q^{r}$ and $|q_{r}\theta-p_{r}| \leq l_{0}E^{-r}$
satisfying $p_{r}q_{r+1} \neq p_{r+1}q_{r}$. Then for any algebraic integers $p$
and $q$ in $\bK$ with $|q| \geq 1/(2l_{0})$, we have
$$
\left| \theta - \frac{p}{q} \right| > \frac{1}{c |q|^{\kappa+1}},
\mbox{ where } c=2k_{0}Q(2l_{0}E)^{\kappa}
\mbox{ and } \kappa = \frac{\log Q}{\log E}.
$$
\end{lem}

\begin{rem}
This is a generalisation of Lemma~2.8 in \cite{CV} to quadratic imaginary fields.
\end{rem}

\begin{proof}
Let $p$, $q$ be algebraic integers in $\bK$ with $|q| \geq 1/(2l_{0}) > 0$.
Choose
$\displaystyle n_{0} = \left\lfloor \frac{\log(2l_{0}|q|)}{\log E} \right\rfloor + 1$.
Since $E>1$ and $2l_{0}|q| \geq 1$, we have $n_{0} \geq 1$. 

It also follows that $\log(2l_{0}|q|)/\log(E) < n_{0}$
and hence for all $n \geq n_{0}$,
\begin{equation}
\label{eq:approx1}
l_{0}E^{-n} < l_{0}E^{-(\log(2l_{0}|q|))/(\log E)} = 1/(2|q|) < 1.
\end{equation}

If we have $q_{n}=0$ for some $n \geq n_{0}$, then from (\ref{eq:approx1}),
$|p_{n}| = |q_{n} \theta - p_{n}| < 1$, which implies that $p_{n}=0$, since
all non-zero algebraic integers in these fields are of absolute value at least
$1$. This contradicts the supposition that $p_{n}q_{n+1} \neq p_{n+1}q_{n}$.
Therefore, $q_{n} \neq 0$ for all $n \geq n_{0}$. 

So, for any $n \geq n_{0}$ with $p/q \neq p_{n}/q_{n}$, we have 
\begin{displaymath}
\left| \theta - \frac{p}{q} \right| 
\geq \left| \frac{p_{n}}{q_{n}} - \frac{p}{q} \right| 
     - \left| \theta - \frac{p_{n}}{q_{n}} \right| 
\geq \frac{1}{|qq_{n}|} - \frac{l_{0}}{E^{n}|q_{n}|}     
> \frac{1}{2|qq_{n}|},
\end{displaymath}
again using (\ref{eq:approx1}) and the fact that $p_{n}q-q_{n}p$ is a non-zero
algebraic integer and hence of absolute value at least $1$ in such fields.

The choice of $n_{0}$ yields
$$
Q^{n_{0}} \leq \exp \left( \frac{\log (2l_{0}|q|) + \log (E)}{\log(E)} \log (Q) \right)
= (2El_{0}|q|)^{\kappa}.
$$

If $p/q \neq p_{n_{0}}/q_{n_{0}}$, then we have
\begin{displaymath}
\left| \theta - \frac{p}{q} \right| 
> \frac{1}{2|q q_{n_{0}}|} 
\geq \frac{1}{2|q|k_{0}Q^{n_{0}}} 
\geq \frac{1}{2k_{0}(2El_{0})^{\kappa}|q|^{\kappa+1}}.  
\end{displaymath}

If $p/q = p_{n_{0}}/q_{n_{0}}$, then we have 
$p/q \neq p_{n_{0}+1}/q_{n_{0}+1}$ and obtain 
\begin{displaymath}
\left| \theta - \frac{p}{q} \right| 
> \frac{1}{2|q q_{n_{0}+1}|} 
\geq \frac{1}{2|q|k_{0}Q^{n_{0}+1}} 
\geq \frac{1}{2k_{0}Q(2El_{0})^{\kappa}|q|^{\kappa+1}}.  
\end{displaymath}
\end{proof}

\begin{lem}
\label{lem:transform}
Let $\bK$ be either $\bQ$ or an imaginary quadratic field and let $\theta \in \bC$.
Suppose that
\begin{displaymath}
\left| q\theta - p \right| > C|q|^{-\kappa},
\end{displaymath}
for some $C$, $\kappa>0$ and all $p, q \in \cO_{\bK}$, the ring of integers of $\bK$,
with $q \neq 0$.

Let $a_{1}$, $a_{2}$, $a_{3}$, $a_{4} \in \cO_{\bK}$ with $a_{1}a_{4}-a_{2}a_{3} \neq 0$
and put $\theta'=(a_{1}\theta+a_{2})/(a_{3}\theta+a_{4})$. Then
\begin{displaymath}
\left| q \theta' - p \right| > \frac{C}{|a_{4}+a_{3}\theta| \left( |a_{3}|(1+|\theta'|)+|a_{1}| \right)^{\kappa}} |q|^{-\kappa},
\end{displaymath}
for the same $C$, $\kappa>0$ and all $p, q \in \cO_{\bK}$ with $q \neq 0$.
\end{lem}

\begin{rem}
This is an explicit version of the results in Section~8 of \cite{Chud},
as well as an extension to include the imaginary quadratic fields. It can be used
to obtain effective irrationality measures for numbers that can be obtained from
$\theta$ by means of fractional transformations.
\end{rem}

\begin{proof}
We can write
\begin{equation}
\label{eq:theta}
\theta = \frac{-a_{4}\theta'+a_{2}}{a_{3}\theta'-a_{1}}.
\end{equation}

Suppose we have $q\theta'-p=\delta$ for some $\delta$. Using this expression,
we can write $\theta'=(\delta+p)/q$ and substituting this expression for $\theta'$
into (\ref{eq:theta}), we find that
\begin{displaymath}
\theta(a_{3}p-a_{1}q)-(a_{2}q-a_{4}p)=-\delta(a_{4}+a_{3}\theta).
\end{displaymath}

From our hypothesis that 
\begin{displaymath}
\left| Q\theta - P \right| > C|Q|^{-\kappa},
\end{displaymath}
for some $C, \kappa>0$ and all $P, Q \in \cO_{\bK}$ with $Q \neq 0$,
we know that
\begin{displaymath}
\left| \delta (a_{4}+a_{3}\theta) \right| > C |a_{3}p-a_{1}q|^{-\kappa}
\end{displaymath}
or
\begin{displaymath}
\left| \delta \right| > \frac{C}{|a_{4}+a_{3}\theta|} |a_{3}p-a_{1}q|^{-\kappa}.
\end{displaymath}

We can assume that $|\delta|<1$. Therefore, $|p| < |q\theta'|+1$.
\begin{displaymath}
|a_{3}p-a_{1}q| < |a_{3}| (1+|\theta' q|) + |a_{1}q| \leq |a_{3}| (|q|+|\theta' q|) + |a_{1}q| = |q| \left( |a_{3}|(1+|\theta'|)+|a_{1}| \right).
\end{displaymath}

Hence
\begin{displaymath}
|q\theta'-p| > \frac{C}{|a_{4}+a_{3}\theta| \left( |a_{3}|(1+|\theta'|)+|a_{1}| \right)^{\kappa}} |q|^{-\kappa},
\end{displaymath}
completing the proof of our lemma.
\end{proof}

\section{Analytic Bounds}

The following lemma is part of Lemma~2.3 in \cite{CV} with one important change.
In Lemma~2.3 of \cite{CV}, we only allowed non-zero values for $x$. We have
removed this condition here as it is not used, or required, in the proof of
Lemma~2.3 in \cite{CV}.

This is important and fortunate, as $x=0$ was actually used to obtain the
theorems in \cite{CV, LPV, TVW}. Therefore, despite the statements in each
of those papers of a result like that Lemma~2.3 which does exclude $x=0$,
the proofs of the theorems in those papers are still sound.

Note that the condition that  $W(x)$ is not a negative real number or zero is
required here with the current proof as it is used in the proof of Lemma~2.2 of
\cite{CV}, a lemma which is used as part of the proof of Lemma~2.3 of \cite{CV}.

\begin{lem}
\label{lem:rembnd}
Let $r$ be a non-negative integer. If $W(x)$ is not a negative real number or zero, 
\begin{eqnarray}
\label{eq:rembnd}
S_{r}(x) 
& = & \left\{ \alpha \left( (x-\beta_{2}) W(x)^{1/n} - (x-\beta_{1}) \right) \right. \nonumber \\
&   & \left. - \left( \beta_{1} (x-\beta_{2}) W(x)^{1/n} - \beta_{2} (x-\beta_{1}) \right) \right\}
	  X_{n,r}^{*}(U(x), Z(x)) \nonumber \\
&   & - (x-\beta_{2})\left( \alpha - \beta_{1} \right) U(x)^{r} R_{1,n,r}(W(x)),
\end{eqnarray}
where 
\begin{displaymath}
R_{m,n,r}(W(x)) = \frac{\Gamma(r+1+m/n)}{\Gamma(m/n) r!}
	     \int_{1}^{W(x)} \left( (1-t)(t-W(x)) \right)^{r} t^{m/n-r-1} \, dt.
\end{displaymath}
\end{lem}

\begin{proof}
We proved this result for $r$ positive in \cite{CV}. It is part of Lemma~2.3
there upon noting that $x-\beta_{2}$, $x-\beta_{1}$,
$\beta_{1} \left( x-\beta_{2} \right)$, $\beta_{2} \left( x-\beta_{1} \right)$
are $6/((n^{2}-1)(\beta_{1}-\beta_{2}))$ times the $a(x)$, $b(x)$, $c(x)$
and $d(x)$ there respectively (see the proof of Lemma~\ref{lem:thue-simp} for
details).

As noted above, we have removed the unnecessary condition that $x$ be non-zero.

So it only remains to consider $r=0$.

Since
\begin{displaymath}
R_{m,n,0}(W(x)) = \frac{\Gamma(1+m/n)}{\Gamma(m/n)} 
	     \int_{1}^{W(x)} t^{m/n-1} \, dt = W(x)^{m/n}-1,
\end{displaymath}
we have
\begin{eqnarray*}
&   & \left\{ \alpha \left( (x-\beta_{2}) W(x)^{1/n} - (x-\beta_{1}) \right) 
	     - \left( \beta_{1} (x-\beta_{2}) W(x)^{1/n} - \beta_{2} (x-\beta_{1}) \right) \right\} \\
&   & \times X_{n,0}^{*}(U(x), Z(x))
      - (x-\beta_{2}) \left( \alpha - \beta_{1} \right) U(x)^{0} R_{1,n,0}(W(x)) \\
& = & \alpha \left\{ (x-\beta_{2})-(x-\beta_{1}) \right\}
- \left\{ \beta_{1}(x-\beta_{2}) - \beta_{2}(x-\beta_{1}) \right\} \\
& = & \alpha Q_{0}(x)-P_{0}(x) = S_{0}(x),
\end{eqnarray*}
and hence (\ref{eq:rembnd}) holds for $r=0$ too.
\end{proof}

Recall that Lemma~\ref{lem:beta-exp} states that for any root, $\alpha$,
of $F(x)$, we can find a value of $x$ such that the first quantity on the
right-hand side of the expression for $S_{r}(x)$ is zero. This is very
important for our needs as otherwise this term would actually grow exponentially
with $r$, whereas we require $S_{r}(x)$ to decrease exponentially quickly
to zero with $r$.

We will show next that the $U(x)^{r}R_{m,n,r}(W(x))$ term approaches 0 exponentially
with $r$.

\begin{lem}
\label{lem:r-upperbnd}
Let $m$, $n$ and $r$ be non-negative integers with $0<m<n$ and $(m,n)=1$.

$({\rm a})$ If either $u$ and $z$ are distinct positive real numbers, or
$u$ and $z$ are complex numbers with $|u|=|z| \neq 0$ and $z/u \neq -1$, then
\begin{displaymath}
\left| u^{r} R_{m,n,r}(w) \right|
\leq 2.38 \left| 1- w^{m/n} \right|
\frac{n \Gamma(r+1+m/n)}{m \Gamma(m/n) r!} 
{ \min \left( | \sqrt{u} - \sqrt{z} |, | \sqrt{u} + \sqrt{z} | \right) }^{2r},
\end{displaymath}
where $w=z/u$.

$({\rm b})$ If $u$ and $z$ are complex numbers with $|1-z/u|<1$, then
$$
\left| u^{r} R_{m,n,r}(w) \right|
< |w^{m/n}-1| \frac{n\Gamma(r+1+m/n)}{m\Gamma(m/n) r!}
\left( \frac{|z-u|^{2}}{4(|u|-|z-u|)} \right)^{r},
$$
where $w=z/u$.
\end{lem}

\begin{proof}
(a) We first consider the case when $u$ and $z$ are distinct positive
real numbers. Using the definition of $R_{m,n,r}(w)$ from Lemma~\ref{lem:rembnd},
we put
\begin{displaymath}
f(t) = \frac{(1-t)(t-w)}{t}.
\end{displaymath}

We find that $({\rm d}/{\rm d}t) f(t)= -\left( t^{2} - w \right) / t^{2}$ and 
that $({\rm d}/{\rm d}t) f(t)=0$ precisely when $t = \pm \sqrt{w}$. Therefore, 
$|f(t)| \leq { \left( 1 - \sqrt{w} \right) }^{2}$ for all $t$ in the closed
interval between $w$ and $1$.

As we saw in the proof of Lemma~\ref{lem:rembnd},
\begin{displaymath}
\int_{1}^{w} t^{m/n-1} \, dt = (n/m) \left( w^{m/n}-1 \right),  
\end{displaymath}
and the lemma follows in this case.

We next consider the case when $u$ and $z$ are complex numbers with $|u|=|z|$.

We proceed similarly to the proof of Lemma~2.5 in \cite{CV}. For $w=e^{i \varphi}$
with $0 < \varphi < \pi$ with $\sqrt{w}=e^{i\varphi/2}$, by Cauchy's theorem, 
\begin{equation}
\label{eq:bigR-integral}
R_{m,n,r}(w) 
= \frac{\Gamma(r+1+m/n)}{\Gamma(m/n)r!}
  \int_{C} { \left( (1-t)(t-w) \right) }^{r} t^{m/n-r-1} dt,   
\end{equation}
where 
\begin{displaymath}
C = \{ t \, | \, t=e^{i \theta}, 0 \leq \theta \leq \varphi \}. 
\end{displaymath}

Put 
\begin{displaymath}
f(t) = \frac{(1-t)(t-w)}{t}  
\hspace{5 mm} \mbox{ and } \hspace{5 mm}
g(t) = t^{m/n-1}. 
\end{displaymath}

Define 
\begin{displaymath}
F(\theta) = { \left| f \left( e^{i \theta} \right) \right| }^{2}, 
\end{displaymath}
so 
\begin{displaymath}
F(\theta) = 4 (1 - \cos \theta)(1 - \cos (\theta - \varphi))  
\hspace{5.0mm} \mbox{ for $0 \leq \theta \leq \varphi$.}  
\end{displaymath}

A simple calculation shows that 
\begin{displaymath}
F'(\theta) = -16 \sin \left( \theta - \frac{\varphi}{2} \right) 
	     \sin \left( \frac{\theta}{2} \right)
	     \sin \left( \frac{\varphi - \theta}{2} \right).   
\end{displaymath}

The only values of $0 \leq \theta \leq \varphi$ with $F'(\theta)=0$ 
are $\theta=0,\varphi/2$ and $\varphi$. It is easy to check that 
\begin{displaymath}
F(\theta) \leq F(\varphi/2) 
= 4 { \left(  1 - \cos \frac{\varphi}{2} \right) }^{2} 
= { \left| 1 - \sqrt{w} \right| }^{4}, 
\end{displaymath}
and hence                                             
\begin{displaymath}
\left| \int_{C} f(t)^{r} g(t) dt \right|
\leq \int_{0}^{\varphi} 
     { \left| f \left( e^{i \theta} \right) \right| }^{r} 
     \left| g \left( e^{i \theta} \right) \right| d \theta 
\leq \varphi { \left| 1 - \sqrt{w} \right| }^{2r}. 
\end{displaymath}

Hence
$$
\left| R_{m,n,r}(w) \right| \leq \varphi \frac{\Gamma(r+1+m/n)}{\Gamma(m/n)r!}
{ \left| 1 - \sqrt{w} \right| }^{2r},
$$
for such $w$.

Note that since $|\varphi| < \pi$, the integrand used in (\ref{eq:bigR-integral})
is continuous over the path of integration.

The same argument can be used to extend this result to all $w$ on the unit circle
using the same definition of the square root.

Notice that $| 1 - \sqrt{w} | \leq | 1 + \sqrt{w} |$ for such $w$, as the real
part of $\sqrt{w}$ is non-negative. Therefore,
$$
\left| \sqrt{u} \left( 1 - \sqrt{w} \right) \right|
\leq \min \left( | \sqrt{u} - \sqrt{z} |, | \sqrt{u} + \sqrt{z} | \right).
$$

Finally, observe that since
$$
|\varphi| \sqrt{1-(m\varphi/n)^{2}/12} < (n/m)\sqrt{2-2\cos(m\varphi/n)}
= (n/m) \left| 1-w^{m/n} \right|,
$$
we have
$$
|\varphi| < \frac{(n/m) \left| 1-w^{m/n} \right|}{\sqrt{1-\pi^{2}/12}}
< 2.38(n/m) \left| 1-w^{m/n} \right|
$$
since $|m\varphi/n| \leq \pi$, so the result holds in this case too.

(b) Following the proof of Lemma~2.4 in \cite{Heu}, we use the change of variables
$t=(1-\lambda)+\lambda w$ to obtain
$$
R_{m,n,r}(w) 
= \frac{\Gamma(r+1+m/n)}{\Gamma(m/n)r!} (w-1)^{2r+1}
  \int_{0}^{1} { \left( \lambda (1-\lambda) \right) }^{r} (1+\lambda(w-1))^{m/n-r-1} d\lambda,
$$

With the estimates $\lambda (1-\lambda) \leq 1/4$ and $|1+\lambda(w-1)| \geq 1-|w-1|$
for $0 \leq \lambda \leq 1$ and $|w-1|<1$, we find that
\begin{eqnarray*}
\left| R_{m,n,r}(w) \right|
& \leq & \frac{\Gamma(r+1+m/n)}{\Gamma(m/n)r!} \left( \frac{|w-1|^{2}}{4(1-|w-1|)} \right)^{r} \\
&      & \times \left| \int_{0}^{1} (w-1) (1+\lambda(w-1))^{m/n-1} d\lambda \right| \\
&  =   & \left| w^{m/n}-1 \right| \frac{n\Gamma(r+1+m/n)}{m\Gamma(m/n)r!} \left( \frac{|w-1|^{2}}{4(1-|w-1|)} \right)^{r},
\end{eqnarray*}
and conclude the proof by substituting $w=z/u$.
\end{proof}

\begin{lem}
\label{lem:polybnd} 
Let $m$, $n$ and $r$ be non-negative integers with $0<m<n$ and $(m,n)=1$.

$({\rm a})$ If either $u$ and $z$ are distinct positive real numbers, or
$u$ and $z$ are complex numbers with $|u|=|z|$, then
\begin{displaymath}
\left| X_{m,n,r}^{*}(z,u) \right|, \left| X_{m,n,r}^{*}(u,z) \right| 
\leq 2 \frac{\Gamma(1-m/n) \, r!}{\Gamma(r+1-m/n)} 
       \left\{ \max \left( \left| \sqrt{u} + \sqrt{z} \right|, \left| \sqrt{u} - \sqrt{z} \right| \right) \right\}^{2r}.
\end{displaymath}

$({\rm b})$ If $u$ and $z$ are complex numbers with $\max \left( |1-z/u|, |1-u/z| \right)<1$, then
\begin{displaymath}
\left| X_{m,n,r}^{*}(z,u) \right|, \left| X_{m,n,r}^{*}(u,z) \right|
\leq 2 \frac{\Gamma(1-m/n) \, r!}{\Gamma(r+1-m/n)} 
     \left\{ 2\left( |u| + |z| \right) \right\}^{r}.
\end{displaymath}
\end{lem}

\begin{proof}
(a) We first consider the case when $u$ and $z$ are distinct positive
real numbers.

For positive integers $r$, from Lemma~5.2 of \cite{Vout1}, we have
\begin{displaymath}
\left| X_{m,n,r}^{*}(z,u) \right|, \left| X_{m,n,r}^{*}(u,z) \right| 
\leq \left( \sqrt{u} + \sqrt{z} \right)^{2r}.
\end{displaymath}

Since $X_{m,n,0}^{*}(u,z)=1$, this also holds for $r=0$.

Since
\begin{equation}
\label{eq:gamma-lb}
\frac{\Gamma(1-m/n) \, r!}{\Gamma(r+1-m/n)} = \frac{r}{r-m/n} \cdots \frac{1}{1-m/n} > 1,
\end{equation}
the desired upper bound holds.

Now we turn to the case when $u$ and $z$ are complex numbers with $|u|=|z|$.

This is an extension of Lemma~2.6 of \cite{CV} to non-negative $r$ and
to any $u$ and $z$ with $w=z/u=e^{i \varphi}$ where $-\pi < \varphi \leq \pi$.

We proceed similarly here and determine the maximum of the function
\begin{displaymath}
F(\theta) = { \left| f \left( e^{i \theta} \right) \right| }^{2} 
	  = 4 (1 - \cos (\theta))(1 - \cos (\theta + \varphi))  
\end{displaymath}
defined there for $0 \leq \theta < 2 \pi$ and fixed $-\pi < \varphi \leq \pi$,
again observing that since $r$ is a positive integer, $f$ is continuous.

Since
\begin{displaymath}
\frac{{\rm d}}{{\rm d}\theta} F(\theta) = -16 \sin \left( \theta + \frac{\varphi}{2} \right) 
	     \sin \left( \frac{\theta}{2} \right)  
	     \sin \left( \frac{- \varphi - \theta}{2} \right),  
\end{displaymath}
the only values of $\theta$ for which $({\rm d}/{\rm d}\theta) F(\theta)=0$
are $\theta = 0, 2\pi$ (so that $\sin (\theta/2)=0$),
$-\varphi/2, \pi - \varphi/2, 2\pi - \varphi/2$ 
(so that $\sin (\theta+\varphi/2)=0$) and $-\varphi, 2\pi - \varphi$
(so that $\sin (-(\theta+\varphi)/2)=0$).

\noindent
$\bullet$ $\theta=0, 2\pi$: $F(\theta)=0$\\
$\bullet$ $\theta=-\varphi/2$: $F(\theta)=4 (1 - \cos (-\varphi/2))(1 - \cos (\varphi/2))=4 (1 - \cos (\varphi/2))^{2}$\\
$\bullet$ $\theta=\pi-\varphi/2$: $F(\theta)=4 (1 - \cos (\pi-\varphi/2))(1 - \cos (\pi+\varphi/2))=4 (1 + \cos (\varphi/2))^{2}$\\
$\bullet$ $\theta=2\pi-\varphi/2$: $F(\theta)=4 (1 - \cos (2\pi-\varphi/2))(1 - \cos (2\pi+\varphi/2))=4 (1 - \cos (\varphi/2))^{2}$\\
$\bullet$ $\theta=-\varphi$: $F(\theta)=4 (1 - \cos (-\varphi))(1 - \cos (0))=0$\\
$\bullet$ $\theta=2\pi-\varphi$: $F(\theta)=4 (1 - \cos (2\pi-\varphi))(1 - \cos (2\pi))=0$

Since $-\pi/2 < \varphi/2 \leq \pi/2$, $0 \leq \cos (\varphi/2) \leq 1$ and hence
the maximum value of $F(\theta)$ is $4 (1 + \cos (\varphi/2))^{2}$. We can write
$|1 + \sqrt{w}|^{2}=(1 + \cos(\varphi/2))^{2} +\sin^{2}(\varphi/2)
=1 + 2\cos(\varphi/2)+\cos^{2}(\varphi/2)+\sin^{2}(\varphi/2)=2 + 2\cos(\varphi/2)$.

Hence $F(\theta) \leq |1 + \sqrt{w}|^{4}$ and following the same steps
as in the remainder of the proof of Lemma~2.6 of \cite{CV}, we find that
\begin{eqnarray*}
\left| X_{m,n,r}^{*}(z,u) \right|
& \leq & 4 |u|^{r} \frac{\Gamma(1-m/n) \, r!}{\Gamma(r+1-m/n)} 
	     \left| 1 + \sqrt{w} \right|^{2r-2}.
\end{eqnarray*}

Since
$$
|\sqrt{u}| \left| 1 + \sqrt{w} \right|
\leq \max \left( \left| \sqrt{u} + \sqrt{z} \right|, \left| \sqrt{u} - \sqrt{z} \right| \right),
$$
\begin{eqnarray*}
\left| X_{m,n,r}^{*}(z,u) \right|
& \leq & \frac{4}{\left| 1 + \sqrt{w} \right|^{2}} \frac{\Gamma(1-m/n) \, r!}{\Gamma(r+1-m/n)}
\max \left( \left| \sqrt{u} + \sqrt{z} \right|, \left| \sqrt{u} - \sqrt{z} \right| \right)^{2r}.
\end{eqnarray*}

Since $w$ is on the unit circle, we can write $1+\sqrt{w}=1+w_{1} \pm \sqrt{1-w_{1}^{2}}i$,
where $0 \leq w_{1} \leq 1$. Hence, $\left| 1+\sqrt{w} \right|^{2}=2+2w_{1} \geq 2$, and so
$$
\frac{4}{\left| 1 + \sqrt{w} \right|^{2}} \leq 2.
$$

It follows that
\begin{eqnarray*}
\left| X_{m,n,r}^{*}(z,u) \right|
& \leq & 2 \frac{\Gamma(1-m/n) \, r!}{\Gamma(r+1-m/n)} 
\max \left( \left| \sqrt{u} + \sqrt{z} \right|, \left| \sqrt{u} - \sqrt{z} \right| \right)^{2r}.
\end{eqnarray*}

To bound $\left| X_{m,n,r}^{*}(u,z) \right|$ from above, we appeal to the fact that
$_{2}F_{1}(-r,-r-m/n,1-m/n,w^{-1})$ is the complex conjugate of $_{2}F_{1}(-r,-r-m/n,1-m/n,w)$,
as shown at the end of the proof of Lemma~2.6 of \cite{CV}.

Finally, for $r=0$, we have $X_{m,n,r}^{*}(z,u)=X_{m,n,r}^{*}(u,z)=1$. From (\ref{eq:gamma-lb}),
the desired upper bound holds for $r=0$.

(b) We prove the upper bound for $X_{m,n,r}^{*}(z,u)$ here, assuming that
$|1-z/u|<1$. The proof for $X_{m,n,r}^{*}(u,z)$ is identical.

We can readily extend the proof of Lemma~2.5 of \cite{Heu} to any $0<m<n$, so
the desired result holds for positive integers, $r$, since
$$
\frac{4e^{2/n}}{\pi} \frac{1}{(2\sqrt{3})^{r+1}} |1-z/u|^{2r+1}|w|^{m/n}<1,
$$
for $r \geq 1$, $n \geq 2$ and $|1-z/u|<1$
and
$2^{r-m/n}(1+|z/u|)^{r+m/n}<\left\{ 2(1+|z/u|) \right\}^{r}$ for such $u$ and $z$.

The proof for $r=0$ is identical to that in (a).
\end{proof}

\begin{lem}
\label{lem:denom}
Suppose that $d$, $m$, $n$ and $r$ are non-negative integers with $d \geq 1$,
$0 < m < n$ and $(m,n)=1$.

$({\rm a})$ Let $D_{m,n,r}$ and $N_{d,n,r}$ be as in the Introduction. Then
$(D_{m,n,r}/N_{d,n,r})X_{m,n,r}(1-dx) \in \bZ[x]$.

Moreover, writing $d=d_{1}d_{2}d_{3}$ where $d_{1}=\gcd(d,n)$,
$d_{2}=\gcd(d/d_{1},n)$ and $d_{3}=d/(d_{1}d_{2})$, we have
$\left( d_{1}^{r} \prod_{p|d_{2}}p^{v_{p}(r!)} \right) | N_{d,n,r}$.

$({\rm b})$ Define
\begin{displaymath}
\mu_{n} = \prod_{\stackrel{\displaystyle p|n}{p, {\rm prime}}} 
p^{1/(p-1)}.  
\end{displaymath}

Then each of the coefficients of the polynomial 
\begin{displaymath}
{2r \choose r} {} _{2}F_{1}(-r,-r \pm m/n, -2r, n \mu_{n}x) 
\end{displaymath}
is a rational integer times non-negative integer powers of $\mu_{n}$.

For $n \geq 3$, $\mu_{n} < 1.94 \log (n)$ and for $n > 420$, $\mu_{n} < 1.18 \log (n)$.

$({\rm c})$ For $n$ in Tables~$1$ and $2$ and putting either
$(\cC_{n},\cD_{n})=(\cC_{1,n},\cD_{1,n})$ or $(\cC_{n},\cD_{n})=(100,\cD_{2,n})$
in those tables, we have
\begin{equation}
\label{eq:num-denom-bnd}
\max \left( 1, \frac{\Gamma(1-m/n) \, r!}{\Gamma(r+1-m/n)},
\frac{n\Gamma(r+1+m/n)}{m\Gamma(m/n)r!} \right)
\frac{D_{m,n,r}}{N_{d,n,r}} < \cC_{n} \left( \frac{\cD_{n}}{\cN_{d,n}} \right)^{r}.
\end{equation}

$({\rm d})$ If $\cN_{d,n}|n$, then $(\ref{eq:num-denom-bnd})$ holds with
$\cC_{n}=1$ and $\cD_{n} = n \mu_{n}$ for all $n \geq 3$ and $\cC_{n}=1$
and $\cD_{n} < 1.18n \log(n)$ for all $n \geq 3$, $n \neq 6$.
\end{lem}

\begin{rem}
In practice, for a particular value of $n$ one should use the results in
Tables~1 and 2 or, for other values of $n$, calculate $n \mu_{n}$ explicitly.
However, the values of $\cC_{n}$ and $\cD_{n}$ in part~(d) will be useful in
obtaining results for arbitrary $n$.
\end{rem}

\begin{rem}
As Wakabayashi states in Remark~3.1 in \cite{Waka2}, it can sometimes beneficial
to have a smaller value of $\cC_{n}$ even at the expense of a somewhat larger
$\cD_{n}$. This is the reason for providing $\cD_{2,n}$ in Tables~1 and 2. For
a given $n$, it is the smallest value of $\cD_{n} \geq \cD_{1,n}$ for which we
can take $\cC_{n} < 100$.
\end{rem}

\begin{rem}
It appears that $n \mu_{n}$ is approximately $\pi/e^{\gamma}$
times the best possible value for $\cD_{n}$. That is
$$
n \mu_{n} \approx \frac{\pi}{e^{\gamma}} \exp \left( \frac{\pi}{\phi(n)} \sum_{j=1,(j,n)=1}^{n/2} \cot \frac{\pi j}{n} \right).
$$
\end{rem}

\begin{rem}
To check the calculations used as part of the proof of part~(c),
we checked the results for all $n$ considered there and all $r \leq 400$ against
calculations done in Maple~8. No differences were found. As well as providing
a test for the correctness of the code used, this also provides further evidence
that Proposition~3.2 of \cite{Vout1} yields exact information on the prime
decomposition of $D_{m,n,r}$.
\end{rem}

\begin{table}[ht]
\begin{center}
\begin{tabular}{||cccc||cccc||} \hline 
$n$    & $\cC_{1,n}$         & $\log \cD_{1,n}$ & $\log \cD_{2,n}$ & $n$   & $\cC_{1,n}$         & $\log \cD_{1,n}$ & $\log \cD_{2,n}$ \\ \hline
$3$    & $2.0 \cdot 10^{7}$  & $0.93$           & $0.97$           & $41$  & $1.2 \cdot 10^{6}$  & $3.37$           & $3.51$           \\ \hline 
$4$    & $4.9 \cdot 10^{6}$  & $1.60$           & $1.64$           & $42$  & $2300$              & $4.81$           & $4.86$           \\ \hline 
$5$    & $8.8 \cdot 10^{9}$  & $1.37$           & $1.42$           & $43$  & $2.4 \cdot 10^{6}$  & $3.42$           & $3.55$           \\ \hline 
$6$    & $35,000$            & $2.75$           & $2.78$           & $44$  & $48,000$            & $4.27$           & $4.34$           \\ \hline 
$7$    & $3.8 \cdot 10^{11}$ & $1.66$           & $1.75$           & $45$  & $31,000$            & $4.33$           & $4.42$           \\ \hline 
$8$    & $2.5 \cdot 10^{8}$  & $2.26$           & $2.39$           & $46$  & $3800$              & $4.26$           & $4.32$           \\ \hline 
$9$    & $4.1 \cdot 10^{11}$ & $2.19$           & $2.27$           & $47$  & $240,000$           & $3.51$           & $3.59$           \\ \hline 
$10$   & $2.6 \cdot 10^{6}$  & $3.02$           & $3.11$           & $48$  & $2.4 \cdot 10^{6}$  & $4.67$           & $4.77$           \\ \hline 
$11$   & $9.7 \cdot 10^{9}$  & $2.06$           & $2.19$           & $49$  & $6400$              & $3.80$           & $3.86$           \\ \hline 
$12$   & $3.9 \cdot 10^{11}$ & $3.18$           & $3.27$           & $50$  & $540$               & $4.58$           & $4.62$           \\ \hline 
$13$   & $1.9 \cdot 10^{13}$ & $2.21$           & $2.31$           & $51$  & $200,000$           & $4.24$           & $4.35$           \\ \hline 
$14$   & $6.9 \cdot 10^{10}$ & $3.24$           & $3.37$           & $52$  & $210,000$           & $4.42$           & $4.46$           \\ \hline
$15$   & $94,000$            & $3.21$           & $3.31$           & $53$  & $13,000$            & $3.64$           & $3.70$           \\ \hline 
$16$   & $3400$              & $2.99$           & $3.09$           & $54$  & $190,000$           & $4.79$           & $4.88$           \\ \hline 
$17$   & $75,000$            & $2.50$           & $2.57$           & $55$  & $1400$              & $4.25$           & $4.34$           \\ \hline 
$18$   & $6.9 \cdot 10^{7}$  & $3.64$           & $3.71$           & $56$  & $2.6 \cdot 10^{6}$  & $4.61$           & $4.71$           \\ \hline 
$19$   & $1.2 \cdot 10^{6}$  & $2.61$           & $2.73$           & $57$  & $52,000$            & $4.35$           & $4.48$           \\ \hline 
$20$   & $14,000$            & $3.60$           & $3.68$           & $58$  & $22,000$            & $4.48$           & $4.56$           \\ \hline 
$21$   & $2.2 \cdot 10^{7}$  & $3.47$           & $3.55$           & $59$  & $1.2 \cdot 10^{7}$  & $3.75$           & $3.96$           \\ \hline 
$22$   & $750,000$           & $3.58$           & $3.66$           & $60$  & $160,000$           & $5.30$           & $5.38$           \\ \hline 
$23$   & $150,000$           & $2.80$           & $2.90$           & $61$  & $14,000$            & $3.79$           & $3.85$           \\ \hline 
$24$   & $140,000$           & $3.93$           & $4.08$           & $62$  & $3500$              & $4.54$           & $4.60$           \\ \hline 
$25$   & $29,000$            & $3.16$           & $3.28$           & $63$  & $14,000$            & $4.61$           & $4.70$           \\ \hline 
$26$   & $6.3 \cdot 10^{6}$  & $3.72$           & $3.83$           & $64$  & $1900$              & $4.44$           & $4.49$           \\ \hline 
$27$   & $840,000$           & $3.38$           & $3.48$           & $65$  & $41,000$            & $4.41$           & $4.51$           \\ \hline 
$28$   & $16,000$            & $3.87$           & $4.00$           & $66$  & $1200$              & $5.22$           & $5.27$           \\ \hline 
$29$   & $27,000$            & $3.02$           & $3.14$           & $67$  & $7400$              & $3.89$           & $3.94$           \\ \hline 
$30$   & $1.4 \cdot 10^{6}$  & $4.53$           & $4.63$           & $68$  & $6800$              & $4.67$           & $4.74$           \\ \hline 
$31$   & $1.3 \cdot 10^{7}$  & $3.09$           & $3.20$           & $69$  & $5100$              & $4.54$           & $4.63$           \\ \hline 
$32$   & $1.1 \cdot 10^{6}$  & $3.70$           & $3.83$           & $70$  & $54,000$            & $5.22$           & $5.31$           \\ \hline 
$33$   & $95,000$            & $3.85$           & $3.94$           & $71$  & $3500$              & $3.95$           & $4.03$           \\ \hline 
$34$   & $4200$              & $3.99$           & $4.06$           & $72$  & $1.2 \cdot 10^{6}$  & $5.09$           & $5.16$           \\ \hline 
$35$   & $890,000$           & $3.85$           & $4.00$           & $74$  & $2200$              & $4.71$           & $4.78$           \\ \hline 
$36$   & $1.8 \cdot 10^{7}$  & $4.36$           & $4.41$           & $75$  & $240,000$           & $4.87$           & $4.99$           \\ \hline 
$37$   & $3200$              & $3.27$           & $3.35$           & $76$  & $1.8 \cdot 10^{10}$ & $4.78$           & $4.86$           \\ \hline 
$38$   & $1100$              & $4.09$           & $4.13$           & $77$  & $11,000$            & $4.55$           & $4.60$           \\ \hline 
$39$   & $40,00$             & $4.00$           & $4.10$           & $78$  & $8.1 \cdot 10^{6}$  & $5.37$           & $5.50$           \\ \hline 
$40$   & $16,000$            & $4.33$           & $4.37$           & $80$  & $39,000$            & $5.07$           & $5.13$           \\ \hline 
\end{tabular}                                                             
\caption{Denominator Bounds: $3 \leq n \leq 80$}
\end{center}
\end{table}

\begin{table}[ht]
\begin{center}
\begin{tabular}{||cccc||cccc||cccc||cccc||} \hline 
$n$   & $\cC_{1,n}$ & $\log \cD_{1,n}$ & $\log \cD_{2,n}$ & $n$   & $\cC_{1,n}$        & $\log \cD_{1,n}$ & $\log \cD_{2,n}$ \\ \hline
$81$  & $170,000$   & $4.57$           & $4.69$           & $132$ & $250$              & $5.99$           & $6.00$           \\ \hline
$82$  & $1000$      & $4.82$           & $4.86$           & $134$ & $50$               & $5.33$           & $5.33$           \\ \hline
$84$  & $35,000$    & $5.58$           & $5.65$           & $138$ & $3400$             & $5.93$           & $5.98$           \\ \hline
$85$  & $1400$      & $4.67$           & $4.70$           & $140$ & $270,000$          & $6.00$           & $6.14$           \\ \hline
$86$  & $2800$      & $4.87$           & $4.95$           & $142$ & $41$               & $5.40$           & $5.40$           \\ \hline
$87$  & $2300$      & $4.77$           & $4.85$           & $143$ & $42$               & $5.17$           & $5.17$           \\ \hline
$88$  & $1700$      & $5.02$           & $5.09$           & $144$ & $1200$             & $5.85$           & $5.90$           \\ \hline
$90$  & $11,000$    & $5.72$           & $5.80$           & $150$ & $2.1 \cdot 10^{6}$ & $6.28$           & $6.41$           \\ \hline
$91$  & $1200$      & $4.71$           & $4.75$           & $154$ & $130$              & $5.95$           & $5.96$           \\ \hline
$92$  & $720$       & $4.97$           & $5.01$           & $156$ & $3400$             & $6.15$           & $6.21$           \\ \hline
$93$  & $1600$      & $4.84$           & $4.90$           & $162$ & $1000$             & $5.99$           & $6.03$           \\ \hline
$94$  & $160$       & $4.96$           & $4.97$           & $163$ & $9.4$              & $4.95$           & $4.95$           \\ \hline
$95$  & $670$       & $4.78$           & $4.83$           & $168$ & $77$               & $6.34$           & $6.34$           \\ \hline
$96$  & $11,000$    & $5.40$           & $5.52$           & $169$ & $8.1$              & $5.16$           & $5.16$           \\ \hline
$98$  & $49,000$    & $5.22$           & $5.35$           & $170$ & $820$              & $6.07$           & $6.10$           \\ \hline
$99$  & $5900$      & $5.03$           & $5.10$           & $174$ & $91$               & $6.18$           & $6.18$           \\ \hline
$100$ & $4300$      & $5.31$           & $5.38$           & $180$ & $89$               & $6.49$           & $6.49$           \\ \hline
$102$ & $240$       & $5.63$           & $5.65$           & $182$ & $54$               & $6.12$           & $6.12$           \\ \hline
$104$ & $600$       & $5.18$           & $5.25$           & $186$ & $1100$             & $6.25$           & $6.29$           \\ \hline
$105$ & $3000$      & $5.55$           & $5.60$           & $190$ & $27$               & $6.19$           & $6.19$           \\ \hline
$106$ & $2400$      & $5.08$           & $5.14$           & $198$ & $7300$             & $6.44$           & $6.48$           \\ \hline
$108$ & $5200$      & $5.53$           & $5.58$           & $210$ & $1900$             & $6.98$           & $7.01$           \\ \hline
$110$ & $2400$      & $5.64$           & $5.70$           & $216$ & $510$              & $6.32$           & $6.34$           \\ \hline
$111$ & $200$       & $5.03$           & $5.05$           & $222$ & $9.5$              & $6.44$           & $6.44$           \\ \hline
$112$ & $1900$      & $5.36$           & $5.40$           & $234$ & $3.0$              & $6.61$           & $6.61$           \\ \hline
$114$ & $1000$      & $5.74$           & $5.78$           & $242$ & $8.7$              & $6.20$           & $6.20$           \\ \hline
$116$ & $360,000$   & $5.21$           & $5.33$           & $243$ & $2.0$              & $5.91$           & $5.91$           \\ \hline
$117$ & $700$       & $5.20$           & $5.23$           & $250$ & $3.5$              & $6.38$           & $6.38$           \\ \hline
$118$ & $7000$      & $5.19$           & $5.24$           & $256$ & $37$               & $6.05$           & $6.05$           \\ \hline
$120$ & $11,000$    & $6.04$           & $6.08$           & $286$ & $4.3$              & $6.60$           & $6.60$           \\ \hline
$121$ & $23$        & $4.76$           & $4.76$           & $326$ & $1.2$              & $6.41$           & $6.41$           \\ \hline
$122$ & $4400$      & $5.23$           & $5.30$           & $338$ & $1.0$              & $6.61$           & $6.61$           \\ \hline
$124$ & $1700$      & $5.28$           & $5.35$           & $360$ & $6.0$              & $7.31$           & $7.31$           \\ \hline
$125$ & $46$        & $4.94$           & $4.94$           & $420$ & $2.3$              & $7.79$           & $7.79$           \\ \hline
$126$ & $22,000$    & $6.01$           & $6.08$           & $432$ & $1.0$              & $7.19$           & $7.19$           \\ \hline
$128$ & $79,000$    & $5.20$           & $5.32$           & $486$ & $1.0$              & $7.36$           & $7.36$           \\ \hline
$130$ & $360$       & $5.80$           & $5.83$           &       &                    &                  &                  \\ \hline 
\end{tabular}                                                             
\caption{Denominator Bounds: $81 \leq n \leq 486$}
\end{center}
\end{table}

\begin{proof}
(a) The first statement, $(D_{m,n,r}/N_{d,n,r})X_{m,n,r}(1-dx) \in \bZ[x]$,
follows immediately from the definitions of these quantities.

The second statement is a more general version of Proposition~5.1 of \cite{Chud}
and we follow his method of proof.

We can write
$$
X_{m,n,r}(1-dx) = \frac{r!n^{r}}{(n-m) \cdots (rn-m)} P_{-m}(dx),
$$
where
\begin{eqnarray*}
P_{-m}(x)
& = & { 2r \choose r} \, _{2}F_{1} (-r,-r-m/n;-2r;x) \\
& = & \sum_{i=0}^{r} \left( \prod_{k=r-i+1}^{r} (kn-m) \right) \frac{1}{i! n^{i}} { 2r-i \choose r} (-x)^{i}.
\end{eqnarray*}
(Notice that this differs from \cite{Chud}. This is due to the fact that
$X_{r}(z)$ and $Y_{r}(z)$ have been incorrectly switched in (4.3), (4.4), (5.2)
and (5.4) of \cite{Chud}).

So
$$
X_{m,n,r}(1-dx)
= \sum_{i=0}^{r} \left( \prod_{k=1}^{r-i} \frac{1}{kn-m} \right) \frac{r! n^{r-i} d_{1}^{i}d_{2}^{i}d_{3}^{i}}{i!} { 2r-i \choose r} (-x)^{i}.
$$

Since $(kn-m,n)=1$ for any integer $k$, it is clear that $d_{1}^{r}$ is a divisor
of the numerator of $X_{m,n,r}(1-dx)$.

Now suppose that $d_{2}>1$ and let $p$ be a prime divisor of $d_{2}$. Then
$p^{i}/p^{v_{p}(i!)}$ is an integer, since $v_{p}(i!) \leq i/(p-1) \leq i$.
Hence we can remove a factor of $p^{v_{p}(r!)}$ from $r!$. Doing so for each
prime divisor of $d_{2}$ completes the proof of part~(a).

(b) The first statement is a slightly stronger and more general version of the
statement of Lemma~2.4 of \cite{CV}, but it is, in fact, what is proved there.
Note that the restriction to $j = \pm 1$ is never used in the proof.

Let $f(x)$ be a positive non-decreasing function for $x \geq 2$ and suppose we
want to show that $\mu_{n} \leq f(n)$. If $n_{1}$ is the largest square-free
divisor of $n$, then $\mu_{n_{1}}=\mu_{n}$. If $\mu_{n_{1}} \leq f(n_{1})$,
then $\mu_{n}=\mu_{n_{1}} \leq f(n_{1}) \leq f(n)$.
So we need only prove $\mu_{n} \leq f(n)$ for square-free $n$.

Furthermore, $g(x)=x^{1/(x-1)}$ is a decreasing function for $x>1$. Therefore,
we can further reduce our consideration to $n=p_{1} \cdots p_{k}$, where $p_{i}$ is the $i$-th prime.

So we can write
\begin{eqnarray*}
\log \mu_{n} & = & \sum_{p \leq p_{k}} \frac{\log p}{p-1}
< \sum_{p \leq p_{k}} \frac{\log p}{p} + \sum_{p} \left( \frac{\log p}{p-1} - \frac{\log p}{p} \right) \\
& < & \sum_{p \leq p_{k}} \frac{\log p}{p} + \sum_{p<P} \frac{\log p}{p(p-1)}
+ \int_{P-1}^{\infty} \frac{\log p}{p-1} dp
- \int_{P-1}^{\infty} \frac{\log p}{p} dp.
\end{eqnarray*}

Following the notation of Section~27.7 of \cite{AbSt} (i.e., letting
$f(x)=-\int_{1}^{x} \log(t)/(t-1) dt$),
$$
\int_{P-1}^{\infty} \frac{\log p}{p-1} dp
- \int_{P-1}^{\infty} \frac{\log p}{p} dp
= \frac{\log^{2}(P-1)}{2} + f(P-1)
- \lim_{z \rightarrow \infty} \left( f(z) + \frac{\log^{2}(z)}{2} \right).
$$

Using the functional relationship $f(x)+f(1/x)=-\log^{2}(x)/2$
(see (27.7.5) in \cite{AbSt}) with $x=1/z$, we see that
$$
\lim_{z \rightarrow \infty} f(z) + \frac{\log^{2}(z)}{2}
= -f(0) = -\frac{\pi^{2}}{6}.
$$

Therefore,
$$
\log \mu_{n} < \sum_{p \leq p_{k}} \frac{\log p}{p} + \sum_{p<P} \frac{\log p}{p(p-1)}
+ \frac{\pi^{2}}{6} + \frac{\log^{2}(P-1)}{2} - \int_{1}^{P-1} \frac{\log(t)}{t-1} dt,
$$
for any prime, $P \geq 3$. With $P=107$, we find that
$$
\log \mu_{n} < \sum_{p \leq p_{k}} \frac{\log p}{p}+0.8.
$$

For $p_{k} \geq 32$, by the Corollary to Theorem~6 of \cite{RS},
\begin{eqnarray*}
\log \mu_{n} & < & \sum_{p \leq p_{k}} \frac{\log p}{p} + 0.8 \\
& < & \log(p_{k}) - 1.33258 + 1/\log(p_{k}) +0.8 < \log (p_{k})-0.244.
\end{eqnarray*}

Recalling that $n=p_{1} \cdots p_{k}$, we have $\log(n)=\theta(p_{k})$, where
$\theta(x)$ is the logarithm of the product of all primes $\leq x$. Hence, from
Theorem~10 of \cite{RS}, for $p_{k} \geq 1427$,
$0.95p_{k} < \theta(p_{k}) = \log (n)$ and so $\log(p_{k})< \log \log (n) - \log (0.95)$.
Thus $\log \mu_{n}<\log \log (n)-0.1927$ and the result holds for such $n$.

A computation for $11 \leq p_{k} < 1427$ shows that
$\log \mu_{n}<\log \log (n)+0.162$, or $\mu_{n} < 1.18 \log(n)$, in this range.

By means of another computation, we find that for $3 \leq n < 2310$,
$\mu_{n}<1.18 \log (n)$ holds except for
$n=3, 4, 6, 10, 12, 18, 30, 42, 60, 210$ and $420$ and that
$\mu_{n} < 1.94 \log(n)$ for these $n$, completing the proof of part~(b).

(c) The basis of the proof of this part of the lemma will be Lemma~3.3 from
\cite{Vout1} and we shall proceed as in the proof of Lemma~5.1 there. However,
to determine how much computation will be needed, we must first find a feasible
value for $\cD_{n}$, so we begin with the analytic bounds.

(c-i) {\bf Analytic Estimates}

{\bf Numerator estimates}

We write $d=d_{1}d_{2}d_{3}$ as in part~(a).

If $d_{2}=1$, then $\cN_{d,n}=d_{1}$ is a divisor of $n$ and, from Lemma~3.5(a)
of \cite{Vout1}, $\cN_{d,n}^{r}$ is a divisor of $N_{d,n,r}$. Hence
$\cN_{d,n}^{r}/N_{d,n,r} \leq 1$.

If $d_{2}>1$, then there exists at least one prime, $p$, such that $p$ contributes
$p^{v_{p}(n)+1/(p-1)}$ to $\cN_{d,n}$. From part~(a), we know that
\begin{eqnarray}
\label{eq:numerbnd-1}
\frac{\cN_{d,n}^{r}}{N_{d,n,r}}
& \leq & \frac{\cN_{d,n}^{r}}{d_{1}^{r}\prod_{p|d_{2}} p^{v_{p}(r!)}}
= \frac{\prod_{p \nmid d_{2}} p^{rv_{p}(d)} \prod_{p|d_{2}} p^{rv_{p}(n)+r/(p-1)}}{d_{1}^{r}\prod_{p|d_{2}} p^{v_{p}(r!)}} \nonumber \\
& = & \prod_{p \mid d_{2}} p^{r/(p-1)-v_{p}(r!)}
\leq \prod_{p \mid n} p^{r/(p-1)-v_{p}(r!)},
\end{eqnarray}
the last equality holding since $\prod_{p\nmid d_{2}}p^{v_{p}(d)} \prod_{p\mid d_{2}}p^{v_{p}(n)}=d_{1}$.

Now $0 \leq r/(p-1)-v_{p}(r!) \leq (\log r)/(\log p)+1/(p-1)$. Therefore,
\begin{equation}
\label{eq:numerbnd-2}
\frac{\cN_{d,n}^{r}}{N_{d,n,r}} \leq r^{\omega(n)}\mu_{n},
\end{equation}
where $\omega(n)$ is the number of distinct prime factors of $n$.

{\bf $\Gamma$-term estimates}

Observe that
$$
\frac{\Gamma(1-m/n)r!}{\Gamma(r+1-m/n)}
= \frac{r}{r-m/n} \cdots \frac{1}{1-m/n} > 1,
$$
for $n \geq 2$.

Similarly,
$$
\frac{n \Gamma(r+1+m/n)}{m\Gamma(m/n)r!}
= \frac{r+m/n}{r} \cdots \frac{1+m/n}{1} > 1,
$$
for $n \geq 1$.

Furthermore, since $(x+a)/x$ is a decreasing function of $x$ for fixed positive $a$,
$$
\frac{n \Gamma(r+1+m/n)}{m\Gamma(m/n)r!}
< \frac{\Gamma(1-m/n)r!}{\Gamma(r+1-m/n)},
$$
so
\begin{equation}
\label{eq:maxbnd}
\max \left( 1, \frac{\Gamma(1-m/n) \, r!}{\Gamma(r+1-m/n)},
\frac{n\Gamma(r+1+m/n)}{m\Gamma(m/n)r!} \right)
= \frac{\Gamma(1-m/n)r!}{\Gamma(r+1-m/n)}.
\end{equation}

Notice that $-\log(1-x)=x+x^{2}/2+x^{3}/3+\cdots$ for $|x|<1$.
Furthermore,
\begin{eqnarray*}
& & (x+x^{2})-\left(x+x^{2}/2+x^{3}/3+\cdots \right)
> x^{2} \left( \frac{1}{2} -\frac{x}{3} \left( 1+x+x^{2}+\cdots \right) \right) \\
& = & x^{2} \left( \frac{1}{2}-\frac{x}{3(1-x)} \right)>0,
\end{eqnarray*}
for $0 < x < 3/5$. Therefore, $-\log(1-x)<x+x^{2}$ for $0 < x < 3/5$, and so
\begin{eqnarray*}
\frac{\Gamma(1-m/n)r!}{\Gamma(r+1-m/n)}
&   =  & \prod_{i=1}^{r} \frac{1}{1-m/(in)}
         = \frac{n}{n-m} \exp \left( \sum_{i=2}^{r} - \log \left( 1-\frac{m}{in} \right) \right) \\
&   <  & \frac{n}{n-m} \exp \left( \sum_{i=2}^{r} \left( \frac{m}{in} + \frac{m^{2}}{(in)^{2}} \right) \right) \\
& \leq & \frac{n}{n-m}
         \exp \left( \int_{1}^{r} \left( \frac{m}{nx} + \frac{m^{2}}{(nx)^{2}}\right) dx \right) \\
&  =   & \frac{n}{n-m} \exp \left( \frac{mnr\log(r) - m^{2} + m^{2}r}{n^{2}r} \right) \\
&  <   & n e^{m^{2}/n^{2}}r^{m/n} \leq (en)r^{(n-1)/n}, 
\end{eqnarray*}
for $r \geq 1$, $n \geq 2$ and $n>m$.

Hence 
$$
\max \left( 1, \frac{\Gamma(1-m/n) \, r!}{\Gamma(r+1-m/n)},
\frac{n\Gamma(r+1+m/n)}{m\Gamma(m/n)r!} \right) \frac{\cN_{d,n}^{r}}{N_{d,n,r}}
< (en)r^{\omega(n)+(n-1)/n}\mu_{n}.
$$
for $n \geq 2$.

We saw in part~(b) of this lemma that $\mu_{n} < 1.18 \log(n)$ for $n>420$, it
follows that $e\mu_{n}<3.21\log(n)$ for such $n$. Computing $e\mu_{n}$ for
$3 \leq n \leq 420$, we find that $e\mu_{n}<5.26\log (n)$ for all $n \geq 3$.

From Th\'{e}or\`{e}me~11 of \cite{Robin}, $\omega(n)<1.3842\log(n)/\log \log(n)$
for $n \geq 3$, so for $n \geq 30$, $\omega(n)+(n-1)/n<1.42\log(n)$.
Computing $\omega(n)+(n-1)/n$ for $3 \leq n < 29$, we find that
$\omega(n)+(n-1)/n<1.59\log(n)$ for all $n \geq 3$.

Therefore,
\begin{eqnarray}
\label{eq:fact}
&   & \max \left( 1, \frac{\Gamma(1-m/n) \, r!}{\Gamma(r+1-m/n)},
      \frac{n\Gamma(r+1+m/n)}{m\Gamma(m/n)r!} \right) \frac{\cN_{d,n}^{r}}{N_{d,n,r}} \\
& < & 5.26 r^{1.59\log(n)} n\log (n), \nonumber
\end{eqnarray}
for $n \geq 3$.

We divide the prime divisors of $D_{m,n,r}$ into two sets, according to their size.
We let $D_{m,n,r}^{(S)}$ denote the contribution to $D_{m,n,r}$ from primes at most
$(nr)^{1/2}$ and let $D_{m,n,r}^{(L)}$ denote the contribution from the remaining,
larger, primes. 

{\bf $D_{m,n,r}^{(S)}$ estimates}

From Lemma~3.3(a) of \cite{Vout1}, we know that 
\begin{displaymath}
D_{m,n,r}^{(S)} \leq \prod_{p \leq (nr)^{1/2}}
p^{\lfloor \log(nr)/(\log(p)) \rfloor}. 
\end{displaymath}

Now $\lfloor x \rfloor \leq 2 \lfloor x/2 \rfloor +1$, so
\begin{eqnarray}
\label{eq:dnrs}
D_{m,n,r}^{(S)}
& \leq & \exp \left\{ 2 \psi \left( \sqrt{nr} \right) + \theta \left( \sqrt{nr} \right) \right\} \nonumber \\
&  <   & \exp \left\{ (2.07766 + 1.01624) \sqrt{nr} \right\}
= \exp \left\{ 3.1 \sqrt{nr} \right\},
\end{eqnarray}
from Theorems~9 and 12 of \cite{RS}.

From (\ref{eq:fact}) and (\ref{eq:dnrs}), we know that 
\begin{eqnarray}
\label{eq:small}
& & \max \left( 1, \frac{\Gamma(1-m/n) \, r!}{\Gamma(r+1-m/n)},
\frac{n\Gamma(r+1+m/n)}{m\Gamma(m/n)r!} \right) \frac{\cN_{d,n}^{r}}{N_{d,n,r}}
D_{m,n,r}^{(S)} \nonumber \\
& < & 5.26 r^{1.59\log(n)} n\log (n) \exp \left\{ 3.1 \sqrt{nr} \right\}.
\end{eqnarray}

{\bf $D_{m,n,r}^{(L)}$ estimates}

For each $n$ in Tables~1 and 2, we let $\epsilon_{n}$ denote the analytic bound
obtained from Table~1 of Ramar\'{e} and Rumely \cite{RR} such that
$| \theta(x;n,k) - x/\phi(n)| < \epsilon_{n}x/\phi(n)$ for $x>10^{10}$,
where $\theta(x;n,k)$ is the logarithm of the product of all primes $p \leq x$
with $p \equiv k \bmod n$ and $\phi(n)$ is Euler's phi function.

From Table~2 of \cite{RR}, we can also find $\epsilon_{n}'$ such that
$|\theta(x;n,k)-x/\phi(n)|<\epsilon_{n}'\sqrt{x}$ for $x \leq 10^{10}$.

Combining these two results, we can find
$X_{0}=(\phi(n)\epsilon_{n}'/\epsilon_{n})^{2}<10^{10}$ such that the analytic
bound $| \theta(x;n,k) - x/\phi(n)| < \epsilon_{n}x/\phi(n)$ holds for $x \geq X_{0}$.
We then compute $\theta(x;n,k)$ for all $x \leq X_{0}$ to find the last value
$X_{1}$ that breaches the analytic bounds of Ramar\'{e} and Rumely for $n$.

Put
$$
\cD_{n,N} = \exp \left\{ \frac{n}{\phi(n)} \left( \sum_{A=0}^{N-1} \sum_{l=1,(l,n)=1}^{n/2} \left( \frac{1+\epsilon_{n}}{nA+l}
		      - \frac{1-\epsilon_{n}}{nA+n-l} \right) + \sum_{l=1,(l,n)=1}^{n/2} \frac{1+\epsilon_{n}}{nN+l} \right) \right\}
$$
and compute $\cD_{n,N}$ for $N \geq 1$ to find the value of $N_{\min}$
that minimises it. We use $\cD_{n,\min}$ to denote this minimum value.

From Lemma~3.3(b) of \cite{Vout1}, we see that for any positive integer $N$
satisfying $nr/(nN+n/2) \geq (nr)^{1/2}$, we have 
\begin{eqnarray*}
D_{m,n,r}^{(L)} 
& \leq & \exp \left\{ \sum_{A=0}^{N-1} \sum_{l=1,(l,n)=1}^{n/2} \left( \theta(nr/(nA+l);n,k_{l})  
		      - \theta(nr/(nA+n-l);n,k_{l}) \right) \right\} \\
& & \times \exp \left\{ \sum_{l=1,(l,n)=1}^{n/2} \theta(nr/(nN+l);n, k_{l}) \right\},
\end{eqnarray*}
where $k_{l} \equiv (-m)l^{-1} \bmod n$.

So, for $r>X_{1}(N_{\min}+1/2)=r_{{\rm comp}}$, $D_{m,n,r}^{(L)} \leq \cD_{n,\min}^{r}$.
Combining this inequality with (\ref{eq:small}) yields 
\begin{eqnarray}
\label{eq:dnr}
& & \max \left( 1, \frac{\Gamma(1-m/n) \, r!}{\Gamma(r+1-m/n)},
\frac{n\Gamma(r+1+m/n)}{m\Gamma(m/n)r!} \right) \frac{\cN_{d,n}^{r}}{N_{d,n,r}}
D_{m,n,r} \nonumber \\
& < & 5.26 r^{1.59\log(n)} n\log (n) \exp \left\{ 3.1 \sqrt{nr} + r \log (\cD_{n,\min}) \right\},
\end{eqnarray}
for $r>r_{{\rm comp}}$.

Therefore we can choose $\cD_{n}$ to be any real number greater than or equal to
$$
\exp \left\{ \frac{\log \left( 5.26 r_{{\rm comp}}^{1.59\log(n)} n\log (n) \right)}
{r_{{\rm comp}}}
+ 3.1 \sqrt{\frac{n}{r_{{\rm comp}}}} \right\}
\cD_{n,\min}.
$$

Then
$$
\max \left( 1, \frac{\Gamma(1-m/n) \, r!}{\Gamma(r+1-m/n)},
\frac{n\Gamma(r+1+m/n)}{m\Gamma(m/n)r!} \right) \frac{\cN_{d,n}^{r}}{N_{d,n,r}}
D_{m,n,r} < \cD_{n}^{r} \leq \cC_{n} \cD_{n}^{r}
$$
for all $r \geq r_{{\rm comp}}$, provided $\cC_{n} \geq 1$. Note that as
$\cD_{n}$ is taken closer to the minimum possible value above, the associated
value of $\cC_{n}$ increases. We will try to strike a balance between the sizes
of these two quantities. Therefore, we will often take $\cD_{n}$ slightly larger
than its minimum possible value here.

We now know $\cD_{n}$ as well as how much computation is required to establish our desired
inequalities for all $r \geq 0$ (a computation which will yield $\cC_{n}$), so we are
ready to describe the required computations.

(c-ii) {\bf Direct Calculations}

First, for each $1 \leq r \leq 1000$, we directly calculate
\begin{displaymath}
\max \left( 1, \frac{\Gamma(1-m/n) \, r!}{\Gamma(r+1-m/n)},
\frac{n\Gamma(r+1+m/n)}{m\Gamma(m/n)r!} \right)
D_{m,n,r},
\end{displaymath}
along with the product over all prime divisors of $n$ of the maximum of $1$ and
$\cN_{p^{v_{p}(n)+1},n}^{r}/N_{p^{v_{p}(n)+1},n,r}$.

(c-iii) {\bf Calculated Estimates}

For $1000 < r \leq r_{{\rm comp}}$, we take the following steps.

(1) Computation of the $\Gamma$ terms in the max term.

(2) Estimation of the numerator.
We calculate the product over all prime divisors of $n$ of $p^{r/(p-1)-v_{p}(r!)}$.

This provides an upper bound for $\cN_{d,n}^{r}/N_{d,n,r}$ over all possible
values of $d$.

This is much faster than calculating the maximum possible value of
$\cN_{d,n}^{r}/N_{d,n,r}$ precisely over all values of $d$. However, if, for a
particular value of $r$, after the denominator steps that follow, this estimate
leads to a large value of $\cC_{n}$, then we do calculate the maximum possible
value of $\cN_{d,n}^{r}/N_{d,n,r}$ precisely.

(3) The computation of the contribution to $D_{m,n,r}$ from the small primes, that
is primes, $p$, satisfying $p \leq (nr)^{1/2}$, using Proposition~3.2 of \cite{Vout1}.

We speed up this part of the calculation, and the following parts, by calculating
and storing the first million primes (the last one being $32,441,957$) and their
logarithms before we start the calculations for any of the $r$'s.

(4) The computation of the contribution to $D_{m,n,r}$ from all primes greater
than $\sqrt{nr}$ and at most $(nr-1)/(nA(r)+1)$ for some non-negative integer
$A(r)$, which depends only on $r$. We use Lemma~3.3(b) of \cite{Vout1} as well
as the cached primes and their logarithms here.

(5) The computation of the contribution to $D_{m,n,r}$ from the remaining larger
primes using the same technique as in \cite{Vout1} of using Lemma~3.3(b) there
and calculating the contributions from each interval and congruence class via the
endpoints of these intervals. The only difference is that here we grew $A(r)$
dynamically over the course of the calculation.

In this manner, we proceeded to estimate the size of the required quantities for 
all $r \leq r_{{\rm comp}}$ to complete the proof of part~(c) of the lemma.

All these calculations were performed using code written in the Java
programming language (JDK 1.5.0.11). The code is available upon request.

(d) Following Chudnovsky \cite{Chud} and defining $\mu_{n,r}=\prod_{p|n} p^{\lfloor r/(p-1) \rfloor}$
(note that we use a somewhat different notation from Chudnovsky to avoid
confusion with $(\mu_{n})^{r}$), from his Lemma~4.2, we know that
$$
\frac{(n-m) \cdots (rn-m)}{r!} \mu_{n,r}
$$
is an integer and that
$\frac{(n-m) \cdots (rn-m)}{r!} \mu_{n,r} X_{m,n,r}(x)$ has integer coefficients.
We will bound this integer from above to obtain our upper bound for $D_{m,n,r}$.

When considering the $\Gamma$ terms in the proof of part~(c), we saw that
$$
\max \left( 1, \frac{\Gamma(1-m/n) \, r!}{\Gamma(r+1-m/n)},
\frac{n\Gamma(r+1+m/n)}{m\Gamma(m/n)r!} \right)
= \frac{\Gamma(1-m/n)r!}{\Gamma(r+1-m/n)}.
$$

Now
$$
\frac{(n-m) \cdots (rn-m)}{r!}=n^{r} \frac{\Gamma(r+1-m/n)}{\Gamma(1-m/n)r!}.
$$

Hence
$$
\max \left( 1, \frac{\Gamma(1-m/n) \, r!}{\Gamma(r+1-m/n)},
\frac{n\Gamma(r+1+m/n)}{m\Gamma(m/n)r!} \right) D_{m,n,r}
\leq n^{r} \mu_{n,r}.
$$

Since $\cN_{d,n}|n$, we saw when considering the numerators in part~(c) that
$\cN_{d,n}^{r}/N_{d,n,r} \leq 1$ and so
\begin{eqnarray*}
\max \left( 1, \frac{\Gamma(1-m/n) \, r!}{\Gamma(r+1-m/n)},
\frac{n\Gamma(r+1+m/n)}{m\Gamma(m/n)r!} \right)
\frac{\cN_{d,n}^{r}D_{m,n,r}}{N_{d,n,r}}
& \leq & n^{r} \mu_{n,r} \leq \left( n \mu_{n} \right)^{r} \\
&  <   & (1.18n \log(n))^{r},
\end{eqnarray*}
for all $n > 420$ from part~(b). In fact, we see that $n \mu_{n} < 1.18n\log(n)$
holds for all $n \geq 3$, except $n=3, 4, 6, 10, 12, 18, 30, 42, 60, 210$ and $420$.
For these excluded values of $n$, we can use the data in Tables~1 and 2 associated
with part~(c), along with some calculation, to show that the desired result holds
and we can take $\cC_{n}=1$ and $\cD_{n}=1.18n \log (n)$ for all $n \geq 3$,
$n \neq 6$.
\end{proof}

\begin{lem}
\label{lem:distinct} 
Let $\beta_{1}, \beta_{2}, P_{r}(x), Q_{r}(x)$ and $F(x)$ 
be defined as in Lemma~$\ref{lem:thue-simp}$ and let $a,b,c$ and 
$d$ be complex numbers satisfying $ad-bc \neq 0$. Define 
\begin{displaymath}
K_{r}(x) = a P_{r}(x) + b Q_{r}(x) 
\mbox{\hspace{5.0mm} and \hspace{5.0mm}} 
L_{r}(x) = c P_{r}(x) + d Q_{r}(x).  
\end{displaymath}

If $\left( x - \beta_{1} \right) \left( x - \beta_{2} \right) F(x) \neq 0$, then
\begin{displaymath}
K_{r+1}(x)L_{r}(x) \neq K_{r}(x)L_{r+1}(x),  
\end{displaymath}
for all $r \geq 0$.
\end{lem}

\begin{proof}
Lemma~2.7 of \cite{CV} states this with our $P_{r}'(x)$ and $Q_{r}'(x)$
in place of $P_{r}(x)$ and $Q_{r}(x)$. Upon noting that our $P_{r}(x)$ and
$Q_{r}(x)$ are constant multiples of $P_{r}'(x)$ and $Q_{r}'(x)$, the result
here holds.
\end{proof}

\section{Proof of Theorem~\ref{thm:general-hypg}} 

We first determine the quantities defined in the Lemma~\ref{lem:thue-simp}.

We have 
\begin{eqnarray*}
W(x) & = & \frac{Z(x)}{U(x)} = \frac{\gamma_{1}}{-\gamma_{2}}
\left( \frac{x-\beta_{1}}{x-\beta_{2}} \right)^{n}.
\end{eqnarray*}

Notice that
$$
W(x) = 1-\frac{U(x)-Z(x)}{U(x)}
\hspace{5.0mm} \mbox{ and } \hspace{5.0mm}
1/W(x) = 1-\frac{Z(x)-U(x)}{Z(x)}.
$$

\subsection{Construction of approximations} 

We now construct our sequences of approximations to $\cA(x)$.

From Lemmas~\ref{lem:thue-simp} and \ref{lem:rembnd}, for $r \geq 0$, we have
\begin{eqnarray*}
Q_{r}(x) & = & \left( x - \beta_{2} \right) X_{n,r}^{*}(Z(x), U(x))
- \left( x - \beta_{1} \right) X_{n,r}^{*}(U(x), Z(x)),  \\ 
P_{r}(x) & = & \beta_{1} \left( x - \beta_{2} \right) X_{n,r}^{*}(Z(x), U(x))
- \beta_{2} \left( x - \beta_{1} \right) X_{n,r}^{*}(U(x), Z(x))
\hspace{3.0mm} \mbox{ and } \\ 
S_{r}(x) & = & -\left( x - \beta_{2} \right) \left( \cA(x) - \beta_{1} \right) U(x)^{r} R_{n, r}(W(x)).
\end{eqnarray*}

These quantities will form the basis for our approximations.

Recalling the definitions of $g$ and $d$ from the statement of Theorem~\ref{thm:general-hypg},
we put $U_{1}(x)=U(x)/g$ and $Z_{1}(x)=Z(x)/g$ and have
\begin{eqnarray*}
X_{n,r}^{*}(U(x), Z(x)) & = & g^{r} X_{n,r}^{*} \left( U_{1}(x), Z_{1}(x) \right) \\
& = & (gZ_{1}(x))^{r}X_{n,r} \left( 1- d\frac{Z_{1}(x)-U_{1}(x)}{dZ_{1}(x)} \right).
\end{eqnarray*}

From Lemma~\ref{lem:denom}(a),
$$
\frac{D_{n,r}}{N_{d,n,r}} X_{n,r} \left( 1- d\frac{Z_{1}(x)-U_{1}(x)}{dZ_{1}(x)} \right)
\in \bZ \left[ \frac{Z_{1}(x)-U_{1}(x)}{dZ_{1}(x)} \right]
$$
and, as a consequence,
$$
Z_{1}(x)^{r} \frac{D_{n,r}}{N_{d,n,r}} X_{n,r} \left( 1- d\frac{Z_{1}(x)-U_{1}(x)}{dZ_{1}(x)} \right)
=\frac{D_{n,r}}{N_{d,n,r}} X_{n,r}^{*}(U_{1}(x), Z_{1}(x))
$$
is an algebraic integer. Hence
\begin{eqnarray*}
\frac{h_{r}D_{n,r}}{g^{r}N_{d,n,r}} X_{n,r}^{*}(U(x), Z(x))
& = & \frac{h_{r} D_{n,r}}{N_{d,n,r}} X_{n,r}^{*}(U_{1}(x), Z_{1}(x)) \\
\mbox{and} 
&   & \\
\frac{h_{r}D_{n,r}}{g^{r}N_{d,n,r}} X_{n,r}^{*}(Z(x), U(x))
& = & \frac{h_{r} D_{n,r}}{N_{d,n,r}} X_{n,r}^{*}(Z_{1}(x), U_{1}(x))
\end{eqnarray*}
are algebraic integers in $\bK(\beta_{1})$ (switching the $U$'s and $Z$'s
in the above argument to prove the latter). Since $x$, $\beta_{1}$ and
$\beta_{2}$ are algebraic integers, it follows that
\begin{equation}
\label{eq:new-prqr-expressions-a}
\frac{h_{r}D_{n,r}}{g^{r}N_{d,n,r}} P_{r}(x), \hspace{3.0mm}
\frac{h_{r}D_{n,r}}{g^{r}N_{d,n,r}} Q_{r}(x) \in \cO_{\bK(\beta_{1})}.
\end{equation}

If $[\bK(\beta_{1}):\bK]=1$, then we let
\begin{equation}
\label{eq:new-prqr-expression-1}
p_{r} = \frac{h_{r}D_{n,r}}{g^{r}N_{d,n,r}} P_{r}(x)
\hspace{3.0mm} \mbox{and} \hspace{3.0mm}
q_{r} = \frac{h_{r}D_{n,r}}{g^{r}N_{d,n,r}} Q_{r}(x).
\end{equation}

If $[\bK(\beta_{1}):\bK]=2$, then by hypothesis, $\beta_{1}$ and $\beta_{2}$
are algebraic conjugates over $\bK$, as are $\gamma_{1}$ and $\gamma_{2}$ and
since $x \in \bK$, $U(x)$ is $-1$ times the algebraic conjugate of $Z(x)$.
Therefore,
$\left( x - \beta_{1} \right) X_{n,r}^{*}(U(x), Z(x))$
is the algebraic conjugate of 
$(-1)^{r}\left( x - \beta_{2} \right) X_{n,r}^{*}(Z(x), U(x))$.
Similarly,
$\beta_{2} \left( x - \beta_{1} \right) X_{n,r}^{*}(U(x), Z(x))$
is the algebraic conjugate of 
$(-1)^{r}\beta_{1} \left( x - \beta_{2} \right) X_{n,r}^{*}(Z(x), U(x))$.

So if $[\bK(\beta_{1}):\bK]=2$ and $r$ is odd, then we let
\begin{equation}
\label{eq:new-prqr-expression-2}
p_{r} = \frac{h_{r}D_{n,r}}{g^{r}N_{d,n,r}} P_{r}(x)
\hspace{3.0mm} \mbox{and} \hspace{3.0mm}
q_{r} = \frac{h_{r}D_{n,r}}{g^{r}N_{d,n,r}} Q_{r}(x).
\end{equation}

If $\bK=\bQ$ and $r$ is even, then
$$
(x-\beta_{1}) \frac{h_{r}D_{n,r}}{g^{r}N_{d,n,r}} X_{n,r}^{*}(U(x), Z(x))
= \frac{a+b\sqrt{t}}{2}
$$
for some choice of rational integers $a$, $b$, $t$ with $t \neq 0$.
Hence
$$
\frac{h_{r}D_{n,r}}{\sqrt{t}g^{r}N_{d,n,r}} Q_{r}(x) = -b \in \bZ.
$$
Similarly,
$$
\frac{h_{r}D_{n,r}}{\sqrt{t}g^{r}N_{d,n,r}} P_{r}(x) \in \bZ.
$$

So if $\bK=\bQ$, $[\bK(\beta_{1}):\bK]=2$ and $r$ is even, then we let
\begin{equation}
\label{eq:new-prqr-expression-3}
p_{r} = \frac{h_{r}D_{n,r}}{\sqrt{t}g^{r}N_{d,n,r}} P_{r}(x)
\hspace{3.0mm} \mbox{and} \hspace{3.0mm}
q_{r} = \frac{h_{r}D_{n,r}}{\sqrt{t}g^{r}N_{d,n,r}} Q_{r}(x).
\end{equation}

If $\bK$ is an imaginary quadratic field and $r$ is even, then
$$
(x-\beta_{1}) \frac{h_{r}D_{n,r}}{g^{r}N_{d,n,r}} X_{n,r}^{*}(U(x), Z(x))
= a+b\sqrt{\tau}
$$
for some $a,b \in \bK$ and where $\tau$ is as in the statement of the Theorem.
Hence
$$
\frac{h_{r}D_{n,r}}{g^{r}N_{d,n,r}} Q_{r}(x) = -2b\sqrt{\tau}
$$
and
$$
\sqrt{\tau}\frac{h_{r}D_{n,r}}{g^{r}N_{d,n,r}} Q_{r}(x) \in \cO_{\bK}.
$$
Similarly,
$$
\sqrt{\tau}\frac{h_{r}D_{n,r}}{g^{r}N_{d,n,r}} P_{r}(x) \in \cO_{\bK}.
$$

So if $\bK$ is an imaginary quadratic field, $[\bK(\beta_{1}):\bK]=2$ and $r$
is even, then we let
\begin{equation}
\label{eq:new-prqr-expression-4}
p_{r} = \sqrt{\tau}\frac{h_{r}D_{n,r}}{g^{r}N_{d,n,r}} P_{r}(x)
\hspace{3.0mm} \mbox{and} \hspace{3.0mm}
q_{r} = \sqrt{\tau}\frac{h_{r}D_{n,r}}{g^{r}N_{d,n,r}} Q_{r}(x).
\end{equation}

These are the numbers we shall use for our approximations. 
We have 
\begin{displaymath}
q_{r} \cA(x) - p_{r} = s_{r}, 
\end{displaymath}
where 
\begin{equation}
\label{eq:new-sr-expression}
s_{r} = t_{r} \frac{h_{r}D_{n,r}}{g^{r}N_{d,n,r}}S_{r}
= -t_{r}\frac{h_{r}D_{n,r}}{N_{d,n,r}}
\left( x - \beta_{2} \right) \left( \cA(x) - \beta_{1} \right) U_{1}(x)^{r} R_{n, r}(W(x)),
\end{equation}
where $t_{r}=1$, $1/\sqrt{t}$ or $\sqrt{\tau}$ depending on values of $p_{r}$
and $q_{r}$ used above and the last equality holds due to the expression for
$\cA(x)$ in the statement of Theorem~\ref{thm:general-hypg} and
Lemma~\ref{lem:rembnd}.

\subsection{Estimates} 

We now want to show that these are ``good'' approximations; we 
do this by estimating $|q_{r}|$ and $|s_{r}|$ from above. 

Since $|t|, |\sqrt{\tau}| \geq 1$, it follows that $|t_{r}| \leq |\sqrt{\tau}|$.
Hence
\begin{eqnarray}
\label{eq:new-qr-bnd}
\left| q_{r} \right|
& \leq & h |\sqrt{\tau}| \frac{D_{n,r}}{N_{d,n,r}} \left\{ \left| \left( x-\beta_{1} \right) X_{n,r}^{*} \left( U_{1}(x), Z_{1}(x) \right) \right| \right. \nonumber \\
&      & \left. + \left| \left( x-\beta_{2} \right) X_{n,r}^{*} \left( Z_{1}(x), U_{1}(x) \right) \right| \right\} \nonumber \\
& \leq & 2h |\sqrt{\tau}| \left( |x-\beta_{1}| + |x-\beta_{2}| \right) \cC_{n} \left( \frac{\cD_{n}}{\cN_{d,n}} \right)^{r} \nonumber \\
& \,   & \times \left\{ \max \left( \left| \sqrt{U_{1}(x)} + \sqrt{Z_{1}(x)} \right|, \left| \sqrt{U_{1}(x)} - \sqrt{Z_{1}(x)} \right| \right) \right\}^{2r},
\end{eqnarray}
from (\ref{eq:new-prqr-expression-1}), (\ref{eq:new-prqr-expression-2}),
(\ref{eq:new-prqr-expression-3}), (\ref{eq:new-prqr-expression-4}), the triangle
inequality, the definitions of $\cC_{n}$, $\cD_{n}$, $h$, $\cN_{d,n}$ and
$Q_{r}(x)$, as well as Lemma~\ref{lem:polybnd}(a).

Furthermore,
\begin{eqnarray}
\label{eq:new-sr-bnd}
\left| s_{r} \right|
&   =  & \left| \frac{t_{r}h_{r}D_{n,r}}{N_{d,n,r}} (x-\beta_{2}) \left( \cA(x)-\beta_{1} \right)  U_{1}(x)^{r} R_{n, r}(W(x)) \right| \nonumber \\
& \leq & 2.4h |\sqrt{\tau}| \left| 1- W(x)^{1/n} \right| |x-\beta_{2}| \left| \cA(x)-\beta_{1} \right| \cC_{n}
         \left( \frac{\cD_{n}}{\cN_{d,n}} \right)^{r} \nonumber \\
&      & \times \left\{ \min \left( \left| \sqrt{U_{1}(x)} + \sqrt{Z_{1}(x)} \right|, \left| \sqrt{U_{1}(x)} - \sqrt{Z_{1}(x)} \right| \right) \right\}^{2r},
\end{eqnarray}
from (\ref{eq:new-sr-expression}), the definition of $h$
as well as Lemma~\ref{lem:r-upperbnd}(a).

Recall that we are only considering $0<W(x)<1$ or $|W(x)|=1$ in Theorem~\ref{thm:general-hypg},
so only part~(a) of Lemmas~\ref{lem:r-upperbnd} and \ref{lem:polybnd} are required here.

We can apply Lemma~\ref{lem:distinct} to see that $p_{r}q_{r+1} \neq p_{r+1}q_{r}$.

From (\ref{eq:new-qr-bnd}) and (\ref{eq:new-sr-bnd}), we can set
\begin{eqnarray*}
k_{0} & = & 2h |\sqrt{\tau}| \left( |x-\beta_{1}| + |x-\beta_{2}| \right) \cC_{n}, \\
l_{0} & = & \max \left( 0.5, 2.4h |\sqrt{\tau}| \left| 1- W(x)^{1/n} \right| \left| x-\beta_{2} \right| \left| \cA(x)-\beta_{1} \right| \cC_{n} \right), \\
E & = & \frac{\cN_{d,n}}{\cD_{n}}
\left\{ \min \left( \left| \sqrt{U_{1}(x)} + \sqrt{Z_{1}(x)} \right|, \left| \sqrt{U_{1}(x)} - \sqrt{Z_{1}(x)} \right| \right) \right\}^{-2} \\
\mbox{and} & & \\
Q & = & \frac{\cD_{n}}{\cN_{d,n}}
\left\{ \max \left( \left| \sqrt{U_{1}(x)} + \sqrt{Z_{1}(x)} \right|, \left| \sqrt{U_{1}(x)} - \sqrt{Z_{1}(x)} \right| \right) \right\}^{2}.
\end{eqnarray*}

Hence we have $\kappa = \log(Q)/\log(E)$ and $c=2k_{0}Q(2l_{0}E)^{\kappa}$ in
Lemma~\ref{lem:approx}.

Since $1/(2l_{0}) \leq 1$, our result follows.

\section{Proof of Theorem~\ref{thm:general-hypg-unitdisk}} 

The proof is identical to the proof of Theorem~\ref{thm:general-hypg} except that
we apply part~(b) of Lemmas~\ref{lem:r-upperbnd} and \ref{lem:polybnd}, rather
than part~(a). With this change, we have
\begin{eqnarray*}
k_{0} & = & 2h |\sqrt{\tau}| \left( |x-\beta_{1}| + |x-\beta_{2}| \right) \cC_{n}, \\
l_{0} & = & \max \left( 0.5, h |\sqrt{\tau}| \left| 1-W(x)^{1/n} \right| \left| x-\beta_{2} \right| \left| \cA(x)-\beta_{1} \right| \cC_{n} \right), \\
E     & = & \frac{\cN_{d,n}}{\cD_{n}}
            \frac{4(|U_{1}(x)|-|Z_{1}(x)-U_{1}(x)|)}{|Z_{1}(x)-U_{1}(x)|^{2}} \\
\mbox{and} & & \\
Q     & = & \frac{\cD_{n}}{\cN_{d,n}}
            2\left( \left| U_{1}(x) \right| + \left| Z_{1}(x) \right| \right).
\end{eqnarray*}

Hence we have $\kappa = \log(Q)/\log(E)$ and $c=2k_{0}Q(2l_{0}E)^{\kappa}$ in
Lemma~\ref{lem:approx}.

Since $1/(2l_{0}) \leq 1$, our result follows.

\section{Proof of Corollary~\ref{cor:cor-1}} 

We first determine the quantities defined in the Lemma~\ref{lem:thue-simp}. 

Put $\beta_{1}=0$, $\beta_{2}=b-a$, $\gamma_{1}=1$, $\gamma_{2}=-(b/a)^{n-1}$
and $x=b$.

We have 
\begin{eqnarray*}
U(x) & = & -\gamma_{2}(x-\beta_{2})^n = \frac{b^{n-1}}{a^{n-1}} \left( x-b+a \right)^{n}, \\
Z(x) & = & \gamma_{1}(x-\beta_{1})^n = x^{n} \mbox{ and } \\
W(x) & = & \frac{Z(x)}{U(x)} = \frac{a^{n-1}}{b^{n-1}} \left( \frac{x}{x-b+a} \right)^{n}.
\end{eqnarray*}

Hence, since $U(b) = ab^{n-1}$, $Z(b) = b^{n}$
and $W(b) = b/a$, using the notation of Lemma~\ref{lem:beta-exp},
\begin{eqnarray*}
\cA(b) & = & \frac{\beta_{1} (b-\beta_{2}) W(b)^{1/n} + \beta_{2} (b-\beta_{1})}
             {(b-\beta_{2}) W(b)^{1/n} + (b-\beta_{1})} \\
& = & \frac{(b-a)b}{-a(b/a)^{1/n} + b}
= -\frac{(b-a)(b/a)^{(n-1)/n}}{1 - (b/a)^{(n-1)/n}}
= \alpha.
\end{eqnarray*}

Since $(a,b)=\cO_{\bK}$, by assumption, we can take $g=b^{n-1}$ and
$h_{r}=h=1$. So we put $U_{1}(x)=a$, $Z_{1}(x)=b$ and $d$ the largest
positive rational integer such that $(a-b)/d$ is an algebraic integer. 

Observe that
$$
\frac{b}{a} \left( \frac{b-a}{\alpha}-1 \right)
=-(b/a)^{1/n},
$$
so if
$$
q_{r}\alpha-p_{r}=s_{r},
$$
then
$$
ap_{r} \left( -(b/a)^{1/n} \right)
-b \left( (b-a)q_{r}-p_{r} \right)
=-b(b-a)\frac{s_{r}}{\alpha}.
$$

We use the $p_{r}$'s and $q_{r}$'s defined in the proof of Theorem~\ref{thm:general-hypg}
(note that they are members of $\cO_{\bK}$). In particular, with the expressions
in this section for the relevant quantities
$$
p_{r} = \frac{h_{r}}{g^{r}} \frac{D_{n,r}}{N_{d,n,r}} P_{r}(b)
= \frac{D_{n,r}}{N_{d,n,r}} (a-b)b X_{n,r}^{*}(a, b).
$$

Note that $p_{r}$ and $b \left( (b-a)q_{r}-p_{r} \right)$ are both divisible
by $b(b-a)$, so we have
$$
\frac{ap_{r}}{b(b-a)} \left( (b/a)^{1/n} \right)
-\left( \frac{p_{r}}{b-a} - q_{r} \right)
=\frac{s_{r}}{\alpha}.
$$

Therefore, by Lemma~\ref{lem:polybnd}(a) along with the definitions of
$\cC_{n}$, $\cD_{n}$ and $\cN_{d,n}$,
we have
\begin{equation}
\label{eq:cor1-qr-bnd}
\left| \frac{ap_{r}}{b(b-a)} \right|
\leq 2|a| \cC_{n} \left( \frac{\cD_{n}}{\cN_{d,n}} \right)^{r}
\left\{ \max \left( \left| \sqrt{a} + \sqrt{b} \right|, \left| \sqrt{a} - \sqrt{b} \right| \right) \right\}^{2r}.
\end{equation}

Similarly, using Lemma~\ref{lem:r-upperbnd}(a),
\begin{eqnarray*}
\left| \frac{s_{r}}{\alpha} \right|
&   =  & \left| \frac{D_{n,r}}{N_{d,n,r}} a U_{1}(b)^{r} R_{n,r}(W(b)) \right| \\
& \leq & 2.38 \left| 1- (b/a)^{1/n} \right| |a| \cC_{n}
\left( \frac{\cD_{n}}{\cN_{d,n}} \right)^{r}
\left\{ \min \left( \left| \sqrt{a} + \sqrt{b} \right|, \left| \sqrt{a} - \sqrt{b} \right| \right) \right\}^{2r}.
\end{eqnarray*}

Next, we require an upper bound for $\left| 1-(b/a)^{1/n} \right|$.

If $b/a \in \bQ$, then $0<b/a<1$ and $0<(a-b)/a<1$.
From the binomial theorem, we have
\begin{eqnarray*}
\left| 1-(b/a)^{1/n} \right|
& = & \left| 1-(1-(a-b)/a)^{1/n} \right| \\
& = & \left| \frac{a-b}{na} \left\{ 1 + \frac{n-1}{2n} \left( \frac{a-b}{a} \right)
 + \frac{(n-1)(2n-1)}{6n^{2}} \left( \frac{a-b}{a} \right)^{2} + \cdots
\right\} \right| \\
& < & \frac{|a-b|}{n|a|} \left\{ 1 + \left( \frac{a-b}{a} \right)
 + \left( \frac{a-b}{a} \right)^{2} + \cdots \right\} \\
& = & \frac{a-b}{n|b|}.
\end{eqnarray*}

If $|b/a|=1$, then we can write $b/a=e^{i\varphi}$ for some $-\pi < \varphi \leq \pi$.
So we have
$$
\left| 1-(b/a)^{1/n} \right| = \sqrt{2-2\cos(\varphi/n)}.
$$

Similarly,
$$
\left| 1-(a/b) \right| = \sqrt{2-2\cos(\varphi)}.
$$

Now $(2/\pi)^{2}\varphi^{2} \leq 2-2\cos(\varphi) \leq \varphi^{2}$ for all
$-\pi < \varphi \leq \pi$. Hence
$$
\left| 1-(b/a)^{1/n} \right| \leq |\varphi/n|
= \pi/(2n) (2/\pi)|\varphi| \leq \pi/(2n) \left| 1-(a/b) \right|.
$$

Therefore,
\begin{equation}
\label{eq:cor1-sr-bnd}
\left| \frac{s_{r}}{\alpha} \right|
\leq 1.25 \left| \frac{a(a-b)}{b} \right| \cC_{n}
     \left( \frac{\cD_{n}}{\cN_{d,n}} \right)^{r}
     \left\{ \min \left( \left| \sqrt{a} + \sqrt{b} \right|, \left| \sqrt{a} - \sqrt{b} \right| \right) \right\}^{2r}.
\end{equation}
since $n \geq 3$.

From (\ref{eq:cor1-qr-bnd}) and (\ref{eq:cor1-sr-bnd}), we can set
\begin{eqnarray*}
k_{0} & = & 2 |a|\cC_{n}, \\
l_{0} & = & 1.25\left| \frac{a(a-b)}{b} \right| \cC_{n}, \\
E & = & \frac{\cN_{d,n}}{\cD_{n}}
\left\{ \min \left( \left| \sqrt{a} + \sqrt{b} \right|, \left| \sqrt{a} - \sqrt{b} \right| \right) \right\}^{-2} \\
\mbox{and} & & \\
Q & = & \frac{\cD_{n}}{\cN_{d,n}}
\left\{ \max \left( \left| \sqrt{a} + \sqrt{b} \right|, \left| \sqrt{a} - \sqrt{b} \right| \right) \right\}^{2}.
\end{eqnarray*}

Hence we have $\kappa = \log(Q)/\log(E)$ and $c=2k_{0}Q(2l_{0}E)^{\kappa}$ in
Lemma~\ref{lem:approx}.

We have $|a/b| \geq 1$ and $|a-b| \geq 1$, since the closest distance between
two algebraic integers in an imaginary quadratic field is $1$. In addition,
$\cC_{n} \geq 1$ (since $D_{n,0}=1$). Therefore, $1/(2l_{0})<1$ and our result follows.

\section{Proof of Corollary~\ref{cor:cor-2}} 

We do not specify a value of $x$ here, so we need only concern ourselves with
determining $d$, $g$, $h_{r}$ and $h$, and hence obtaining expressions for
$E$, $Q$ and $c$.

$\bullet$ $g$

Using the definitions of $g$ and the $g_{i}$'s in Corollary~\ref{cor:cor-2},
we will show that $(U(x)\sqrt{g_{3}/g_{2}}/g_{1})^{2}$ is an algebraic integer
(and in $\bQ(\sqrt{t})$). Since $g_{4} \in \bZ$, it will follow that
$U(x)/g=\sqrt{g_{4}} U(x)\sqrt{g_{3}/g_{2}}/g_{1}$ is also an algebraic integer.

Writing
$$
\frac{U^{2}(x)g_{3}}{g_{1}^{2}g_{2}}
= \frac{g_{3}}{4}
\left\{ \left( \frac{u_{1}}{g_{1}g_{2}} \right)^{2}g_{2}
+ \left( \frac{u_{2}}{g_{1}} \right)^{2} \frac{t}{g_{2}} \right\}
+ \frac{u_{1}}{g_{1}g_{2}} \frac{u_{2}}{g_{1}}\frac{g_{3}}{2} \sqrt{t} \in \bQ \left( \sqrt{t} \right),
$$
we have
$$
\mbox{Trace} \left( \frac{U^{2}(x)g_{3}}{g_{1}^{2}g_{2}} \right)
= \frac{g_{3} \left( u_{1}^{2}+u_{2}^{2}t \right)}{2g_{1}^{2}g_{2}}
\hspace{0.5mm}\mbox{ and }\hspace{1.0mm}
\mbox{Norm} \left( \frac{U^{2}(x)g_{3}}{g_{1}^{2}g_{2}} \right)
= \frac{g_{3}^{2} \left( u_{1}^{2}-u_{2}^{2}t \right)^{2}}{16g_{1}^{4}g_{2}^{2}}.
$$

If both of these quantities are rational integers, then the minimal polynomial
of $U^{2}(x)g_{3}/(g_{1}^{2}g_{2})$ over $\bQ$ will be monic with rational integer
coefficients and hence $U^{2}(x)g_{3}/(g_{1}^{2}g_{2})$ is an algebraic integer.

From the definitions of $g_{1}$ and $g_{2}$, $t/g_{2}$, $u_{1}/(g_{1}g_{2})$
and $u_{2}/g_{1}$ are all rational integers, so $\left( u_{1}^{2}+u_{2}^{2}t \right)
/ \left( g_{1}^{2}g_{2} \right)$ and
$\left( u_{1}^{2}-u_{2}^{2}t \right)^{2} / \left( g_{1}^{4}g_{2}^{2} \right)$
are both rational integers.

If $t \equiv 1 \bmod 4$ and $(u_{1}-u_{2})/g_{1} \equiv 0 \bmod 2$, then
$$
\frac{u_{1}^{2}}{g_{1}^{2}} + \frac{u_{2}^{2}}{g_{1}^{2}}t
=
\left\{
\begin{array}{ll}
0 \bmod 4 & \mbox{ if $u_{1}/g_{1} \equiv u_{2}/g_{1} \equiv 0 \bmod 2$} \\
2 \bmod 4 & \mbox{ if $u_{1}/g_{1} \equiv u_{2}/g_{1} \equiv 1 \bmod 2$},
\end{array}
\right.
$$
while
$$
\frac{u_{1}^{2}}{g_{1}^{2}} + \frac{u_{2}^{2}}{g_{1}^{2}}t
= 0 \bmod 4,
$$
if $t \equiv 3 \bmod 4$ and $(u_{1}-u_{2})/g_{1} \equiv 0 \bmod 2$.

In both cases, $t$ is odd, so $g_{2}$ is also odd and thus
$\mbox{Trace} \left( U^{2}(x)g_{3}/(g_{1}^{2}g_{2}) \right) \in \bZ$.

If neither of these conditions holds (i.e., if we are in the ``otherwise''
case of the definition of $g_{3}$), then $g_{3}=4$ and again
$\mbox{Trace} \left( U^{2}(x)g_{3}/(g_{1}^{2}g_{2}) \right) \in \bZ$.

We proceed in a similar way to show that
$\mbox{Norm} \left( U^{2}(x)g_{3}/(g_{1}^{2}g_{2}) \right) \in \bZ$.

If $t \equiv 1 \bmod 4$ and $(u_{1}-u_{2})/g_{1} \equiv 0 \bmod 2$, then
$$
\frac{u_{1}^{2}}{g_{1}^{2}} - \frac{u_{2}^{2}}{g_{1}^{2}}t
= 0 \bmod 4,
$$

If $t \equiv 3 \bmod 4$ and $(u_{1}-u_{2})/g_{1} \equiv 0 \bmod 2$, then
$$
\frac{u_{1}^{2}}{g_{1}^{2}} - \frac{u_{2}^{2}}{g_{1}^{2}}t
=
\left\{
\begin{array}{ll}
0 \bmod 4 & \mbox{ if $u_{1}/g_{1} \equiv u_{2}/g_{1} \equiv 0 \bmod 2$} \\
2 \bmod 4 & \mbox{ if $u_{1}/g_{1} \equiv u_{2}/g_{1} \equiv 1 \bmod 2$}.
\end{array}
\right.
$$

Since in both cases $t$ is odd, $g_{2}$ is also odd, so
$\mbox{Norm} \left( U^{2}(x)g_{3}/(g_{1}^{2}g_{2}) \right) \in \bZ$.

If neither of these conditions holds, then $g_{3}=4$ and again
$$
\mbox{Norm} \left( U^{2}(x)g_{3}/(g_{1}^{2}g_{2}) \right) \in \bZ.
$$

Since $Z(x)$ is $-1$ times the algebraic conjugate of $U(x)$,
$Z(x)/g$ is an algebraic integer as well.

$\bullet$ $h_{r}$

Since $g^{2} \in \bQ$, we can take $h_{r}=1$ for $r$ even. However, if, for
example, $g_{2}$ is a proper divisor of $t$ or $g_{3}=2$, then $g$
need not be a perfect square. Since $g_{2}g_{3}g_{4}/\core(g_{2}g_{3}g_{4})$
is a perfect square, we can take $h_{r}=\sqrt{\core(g_{2}g_{3}g_{4})}$ for
$r$ odd. From the definition of $g_{4}$,
$g_{4}|\core(g_{2}g_{3})$, so $\core(g_{2}g_{3}g_{4}) \leq \core(g_{2}g_{3})$.
Since $g_{2}|t$ and $g_{3}=1$, $2$ or $4$, $\core(g_{2}g_{3}) \leq 2t$. Hence
$h_{r} \leq \sqrt{2t}$ for $r$ odd.

$\bullet$ $d$

Since $U(x)-Z(x)=u_{1}$, our definition of $d$ is correct.

$\bullet$ $E$ and $Q$

Since $U(x)=(u_{1}+u_{2}\sqrt{t})/2$, we have $Z(x)=(-u_{1}+u_{2}\sqrt{t})/2$.
Hence,
$$
\left( \sqrt{U(x)} \pm \sqrt{Z(x)} \right)^{2}
=U(x)+Z(x) \pm 2\sqrt{U(x)Z(x)}
=u_{2}\sqrt{t} \pm \sqrt{u_{2}^{2}t-u_{1}^{2}},
$$
giving rise to our expressions for $E$ and $Q$.

$\bullet$ $c$

From our determination of $h_{r}$ above, we can let $h=\sqrt{|2t|}$. Since
$\bK=\bQ$ here, we have $\tau=1$ and hence
take $c$ to be
\begin{eqnarray*}
& & 4 \sqrt{|2t|} \left( |x-\beta_{1}| + |x-\beta_{2}| \right) \cC_{n} Q \\
& & \times
\left( \max \left( 1, 5 \sqrt{|2t|} \left| 1- W(x)^{1/n} \right| |x-\beta_{2}| \left| \cA(x)-\beta_{1} \right| \cC_{n}E \right) \right)^{\kappa}.
\end{eqnarray*}

\subsection*{Acknowledgements}

The author thanks Michel Waldschmidt for initially bringing this method to the
author's attention as well as his advice and support over the years and also
thanks Gary Walsh for his encouragement to resume work in this area.

In addition, the author is very grateful to the referee for their very careful
reading of the manuscript. The corrections, clarifications and suggestions they
provided improved this article considerably.

\end{document}